\documentclass[10pt]{amsart}

\usepackage{a4wide, amstext, amsmath, amssymb, graphicx, array, epsfig, psfrag}
\usepackage{amstext, amsmath, amssymb, graphicx}
\usepackage{todonotes}
\usepackage{eucal,amsfonts,mathrsfs,amscd,bm,tcolorbox}

\newtheorem{theorem}{Theorem}[section]

\newtheorem{corollary}[theorem]{Corollary}
\newtheorem{proposition}[theorem]{Proposition}
\theoremstyle{definition}

\theoremstyle{remark}
\newtheorem{remark}[theorem]{Remark}

\newcommand{\bfzero}{\boldsymbol 0}
\newcommand{\bfbeta}{\boldsymbol \beta}
\newcommand{\bfchi}{\boldsymbol \chi}
\newcommand{\bfomega}{\boldsymbol \omega}
\newcommand{\bff}{\boldsymbol f}

\newcommand{\bfn}{\boldsymbol n}

\newcommand{\bfL}{\boldsymbol L}
\newcommand{\bfx}{\boldsymbol x}

\newcommand{\bfe}{\boldsymbol e}

\newcommand{\bft}{\boldsymbol t}
\newcommand{\bfu}{\boldsymbol u}
\newcommand{\bfv}{\boldsymbol v}
\newcommand{\bfw}{\boldsymbol w}

\newcommand{\bfV}{\boldsymbol V}
\newcommand{\bfW}{\boldsymbol W}

\newcommand{\bfpsi}{\boldsymbol \psi}

\numberwithin{equation}{section}

\newcommand{\jump}[1]{[\![#1]\!]}

\newcommand\calL{\mathcal{L}}
\newcommand\calT{\mathcal{T}}
\newcommand\calF{\mathcal{F}}
\newcommand{\bcalV}{\boldsymbol{\mathcal V}}
\newcommand\bZ{\boldsymbol{Z}}
\newcommand\bz{\boldsymbol{z}}
\newcommand\bfalpha{\boldsymbol{\alpha}}
\newcommand{\dive}{{\ensuremath\mathop{\mathrm{div}\,}}}
\newcommand{\curl}{{\ensuremath\mathop{\mathrm{curl}\,}}}

\newcommand{\pol}{\mathbb{P}}

\newcommand{\bl}{\bigl\langle}
\newcommand{\br}{\bigr\rangle}

\newcommand{\vertiii}[1]{{\left\vert\kern-0.25ex\left\vert\kern-0.25ex\left\vert #1 
    \right\vert\kern-0.25ex\right\vert\kern-0.25ex\right\vert}}

\usepackage{graphicx}

\title[Stabilizing the vorticity for the Oseen equation]{
A pressure-robust discretization of Oseen's equation using stabilization in the vorticity equation}

\author[N.~Ahmed]{Naveed Ahmed}
\address{Department of Mathematics and Natural Sciences, Gulf University for Science and Technology, Mubarak Al-Abdullah Area/West Mishref, Kuwait}
                    \email{ahmed.n@gust.edu.kw}

\author[G.R.~Barrenechea]{Gabriel R. Barrenechea}
\address{Department of Mathematics and
    Statistics, University of Strathclyde, 26 Richmond Street,
    Glasgow, G1 1XH United Kingdom}
\email{gabriel.barrenechea@strath.ac.uk}

\author[E.~Burman]{Erik Burman}
\address{Department of Mathematics, University College London, London, UK-WC1E  6BT, United Kingdom}
\email{e.burman@ucl.ac.uk}

\author[J.~Guzm\'an]{Johnny Guzm\'an}
\address{Division of Applied Mathematics
Brown University
Box F
182 George Street
Providence, RI 02912}
\email{johnny\_guzman@brown.edu}

\author[A.~Linke]{Alexander Linke}
\address{ Weierstrass Institute for Applied Analysis and Stochastics (WIAS), Mohrenstr. 39, 10117 Berlin, Germany}
\email{linke@wias-berlin.de}

\author[C.~Merdon]{Christian Merdon} 
\address{Weierstrass Institute for Applied Analysis and Stochastics (WIAS), Mohrenstr. 39, 10117 Berlin, Germany} 
\email{merdon@wias-berlin.de}

\begin{document}

\maketitle

\begin{abstract}
Discretization of Navier-Stokes' equations using
pressure-robust finite element methods
is considered for the high Reynolds number regime. To counter
oscillations due to dominating convection
we add a stabilization based on  a  bulk term in the form of a
residual-based least squares stabilization of the vorticity equation supplemented
by a  penalty term  on (certain components of) the gradient jump over 
the elements faces.
Since  the stabilization is based on the vorticity equation, it is independent
of the pressure gradients, which makes it pressure-robust. Thus,  we 
 prove pressure-independent error estimates in the linearized case, known as Oseen's problem.
In fact, we prove an $O(h^{k+\frac12})$ error estimate in the $L^2$-norm that is
known to be the best that can be expected for this type of problem.
Numerical examples are provided that, in addition to confirming the theoretical results, 
show that the present method compares favorably to the  classical residual-based
SUPG stabilization.
\end{abstract}


\noindent AMS class: 65N30;  65N12;  76D07

\section{Introduction}
In recent years, it has been observed that the saddle point
structure of the incompressible Navier--Stokes equations
\begin{align} \label{eq:nse}
 \bfu_t - \mu \Delta \bfu + \left (\bfu \cdot \nabla \right ) \bfu
  + \nabla p  & = \bff\,,  \\
\dive \bfu & = 0\,, \nonumber
\end{align}
induces, besides the fulfillment of the well-known
discrete Ladyzhenskaya–-Babu\v{s}ka–-Brezzi (LBB) condition \cite{BoffiBrezziFortin13,gr:book:1986},
 a second fundamental challenge \cite{JLMNR2017}. 
This second challenge is briefly described as follows:
since the pressure acts as a Lagrangian multiplier for the divergence constraint,
the pressure gradient $\nabla p$ will always balance any occurring, unbalanced
gradient field in the momentum balance.
Thus, gradient fields in the momentum balance do only change
$\nabla p$, but not the velocity $\bfu$, which leads
to the existence of certain equivalence classes of
forces  --- and a corresponding seminorm ---
that determine the solution structure of the problem \cite{gls:2019}.
The purpose of this work is to investigate
the relation of this second challenge to the
question of how to stabilize dominant advection in high Reynolds number flows.

Space discretizations that remain accurate in the presence of dominant gradient fields
in the momentum balance --- leading to strong pressure gradients ---
have recently triggered a notable research activity \cite{gls:sinum:2020,rw:cmame:2020,lms:nm:2019,kr:siam:2019,fkls:ijnmf:2019,vz:sinum:2019,CK18,GN18,GS19,alm:siam:2018,lls:siam:2018,SLLL18,sl:jnm:2017,lmt:esaim:2016,EH13,eh:jcp:2013} and have been called {\em pressure-robust} \cite{LinkeMerdon16,alm:cmam:2018}.
This concept explains how these equivalence classes of
forces and the special role of gradient-type forces
affect the {\em notion of dominant advection} in Navier--Stokes flows.
Starting from the idea of pressure-robustness, in this work we  propose a novel discrete stabilization operator for
Navier--Stokes flows that uses only the vorticity equation, and not
 the entire momentum equation. As a model we consider the linear steady-state Oseen equation discretized by means of
an inf-sup stable pair of spaces using $H^1$-conforming velocities of polynomials of order $k$ and pressures
of order $k-1$. To this discrete system,  we add a GLS-type term to the formulation involving the vorticity equation.  One of the main results
emerging from the analysis of the method is that we are able to prove the following error estimate in the convection dominated regime:
\begin{equation}\label{main-estimate}
\| \bfu - \bfu_h\|_{L^2} \leq C h^{k+\frac12} | \bfu|_{H^{k+1}}.
\end{equation}
To the best of our knowledge, this closes a
disturbing gap in the theory of mixed finite elements
that had not been overcome yet. In fact,  
when $H^1$-conforming velocities are used,
the same {\em convergence order}
for the velocity error in the $\bfL^2$ norm had
been achieved for equal order interpolation methods only (see, e.g., \cite{FF92,BB06,BFH06,BBJL07,Cod08}) ---
though at the price of an additional dependency of the velocity error on the continuous
pressure $p$, i.e., giving up pressure-robustness.
Interestingly, whenever the degree of spaces is different (such as in the present case) only in the very recent
paper \cite{BBG20} such an estimate has been proven for an incompressible flow problem, at the cost
of giving up $H^1$-conformity. In fact, the spaces used in \cite{BBG20} were only $H(div)$-conforming.

The main reason why an estimate such as \eqref{main-estimate} has not been obtained using inf-sup stable
elements supplied with classical stabilization mechanisms is linked to the pressure gradient. In fact, 
when SUPG stabilization is used, the
pressure must be included in the stabilizing terms for consistency, and the approximation of the pressure, being of 
a lower order than the one for velocity, prevents from proving \eqref{main-estimate}.
Symmetric stabilization methods, such as Continuous Interior Penalty, Local Projection Stabilization,
and Orthogonal Subscales Method  have been successfully used for scalar convection-diffusion equations.
When one of these methods is applied to the Oseen equation with inf-sup stable elements,  the stabilization is independent of the pressure. So, in principle, the
application of the same analysis from the convection-diffusion equation to the Oseen equation seems achievable. Nevertheless, a more detailed inspection
shows that their stability and convergence relies on 
 orthogonality properties of some interpolant and
stabilization of the orthogonal complement of the convective
term. Consistency is obtained since this orthogonal complement tends to
zero at an optimal rate under refinement.
Nevertheless,  even when pressure-robust spaces are used,
similar orthogonality can  not be exploited for the
vector-valued Oseen's problem since the convection term
itself is not divergence-free in general \cite{BL08}.
So, the extension of the existing analysis for a scalar convection-diffusion equation can not
be carried out unless the pressure gradients are eliminated. Based on this observation, in this work we
add a stabilizing term that penalizes the equation for the vorticity, where pressure
gradients are naturally absent, and no extra properties  of the convective term are required.

The structure of the manuscript is as follows: the introduction is completed by two short sections, one regarding
the motivation and background for the new method introduced in this work, and one containing preliminary results
about vector potentials for divergence-free functions and their regularity.
Then,  in Section \ref{sec:fem:spaces} we introduce the finite element method used in this work, along with
various examples of finite element spaces that are appropriate for its use.
In Section \ref{sec:numerical:analysis}, we deliver a detailed numerical
analysis. We achieve optimal convergence orders for the discrete velocity including,
as stated earlier, the first $O(h^{k+\frac{1}{2}})$ error estimate for the velocity in the
$L^2$-norm for the convection-dominated regime. A supercloseness result for the discrete pressure,
typical for pressure-robust discretizations, is also derived.
In Section \ref{sec:numerics}, we will provide and discuss the results of testing the present method
for different benchmark problems. The benchmarks cover both extreme cases, where the
convection term is a gradient field or a divergence-free vector field.
Also, the general case is treated, where the convection term is
a sum of a gradient field and a divergence-free field.
The new LSVS stabilization is compared to a Galerkin discretization and a
SUPG stabilization applied to the same pairs of finite element methods.
The numerical results show the improvement provided by the new stabilized
method over both the Galerkin and SUPG methods.
Finally, some conclusions are drawn in Section~\ref{concl}.

\subsection{Background, motivation, and notations}
We start by setting the notation to be used throughout. We will use standard notation for Lebesgue and Sobolev  spaces, in line with, e.g., \cite{gr:book:1986}.
In particular, for a domain $D\subseteq\mathbb{R}^d, d=2,3$, and $q\in [1,+\infty]$, $L^q(D)$ will denote the space of  measurable functions such that its
$q^{\rm th}$ power is integrable in $D$ (for $q<+\infty$) and essentially bounded in $D$ (when $q=+\infty$). The space $L^q_0(D)$ denotes the space of functions
in $L^q(D)$ with zero mean value in $D$.  Its norm will be denoted by $\|\cdot\|_{q,D}^{}$
(except when $q=2$, in which case we denote the norm by $\|\cdot\|_{0,D}^{}$). In addition, the inner product in $L^2(D)$ will be denoted by $(\cdot,\cdot)_D^{}$.
For $k\ge 0$ the space $W^{k,q}(D)$ denotes all generalized functions that belong to $L^q(D)$ with distributional derivatives up to the $k^{\rm th}$ order
belonging to $L^q(D)$. We will denote its norm (seminorm) by $\|\cdot\|_{k,q,D}^{}$ ($|\cdot|_{k,q,D}^{}$). When $q=2$, $W^{k,2}(D)=H^k(D)$, and
its norm (seminorm) is denoted by $\|\cdot\|_{k,D}^{}\, (|\cdot|_{k,D}^{})$. The space $W^{k,q}_0(D)$ ($H^k_0(D)$) denotes the closure of $C_0^\infty(D)$ in
$W^{k,q}(D)\, (H^k(D))$. The space $H^{-1}(D)$ denotes the dual of $H^1_0(D)$ with respect to the  inner product in $L^2(D)$,  the corresponding duality pairing
will be denoted by $\bl\cdot,\cdot\br_D^{}$, and the associated norm is denoted by $\|\cdot\|_{-1,D}^{}$. The vector-valued counterpart of a space $X$ will be 
denoted simply by $X^d$, and the same notation will be used for inner products, norms, and duality pairing.

In order to motivate our new stabilization approach,
we now reflect  on the notion
of {\em dominant advection} for the incompressible Navier--Stokes equations.
We set the problem in a bounded, polyhedral, contractible domain $\Omega\subseteq \mathbb{R}^d, d=2,3,$ with Lipschitz continuos
boundary $\partial\Omega$. In addition, we define the space of divergence-free functions in $\Omega$ as follows
\begin{equation}\label{div-free-space}
\bcalV(\Omega) := \{\bfv \in H_0^1(\Omega)^d \mbox{ such
  that } \dive \bfv = 0 \mbox{ in } \Omega\}\,,
\end{equation}
and  regard the following weak formulation
under homogeneous Dirichlet boundary conditions
with time-independent test functions:
search for $\bfu(t) \in \bcalV(\Omega)$ such that for all $\bfv \in \bcalV(\Omega)$ the following holds
\begin{equation} \label{eq:weak:nse}
  \frac{\mathrm{d} }{\mathrm{dt}} (\bfu(t), \bfv)_\Omega^{} + \mu (\nabla \bfu(t), \nabla \bfv)_\Omega^{}
   + ((\bfu(t) \cdot \nabla) \bfu(t), \bfv)_\Omega^{} = \bl\bff(t), \bfv \br_\Omega^{}\,,
\end{equation}
in the sense of distributions in $\mathcal{D}'(]0,T[)$, with
$\bfu(0) = \bfu_0$ in $\bcalV(\Omega)'$ fulfilled in the weak sense, see \cite{bf:book:2013}.
We remark that the weak formulation is pressure-free, avoiding issues with
a possible low regularity of the pressure field in the transient nonlinear setting.

Let now  $\bff_1,\bff_2  \in H^{-1}(\Omega)^d$  be two forcings
differing only by a gradient field, i.e., $\bff_1 - \bff_2 = \nabla \phi$
with $\phi \in L^2(\Omega)$. We interpret these forcings as functionals in $\bcalV(\Omega)'$ and
compute for arbitrary $\bfv \in \bcalV(\Omega)$
$$
  \bl \bff_1, \bfv \br_\Omega^{} - \bl \bff_2, \bfv \br_\Omega^{} \,  = \,
  \bl \bff_1 - \bff_2, \bfv \br_\Omega^{}  \, = \, -( \phi, \nabla \cdot \bfv)_\Omega^{} = 0.
$$
Thus, $\bff_1$ and $\bff_2$ are identical if they
are regarded as functionals in $\bcalV(\Omega)'$.
This leads to the fundamental observation
that $\bff_1$ and $\bff_2$ are {\em velocity-equivalent}
in the sense that they  induce the very same velocity solution in \eqref{eq:weak:nse}.
A difference between $\bff_1$ and $\bff_2$ can only be recognized
in the original equations
\eqref{eq:nse}, where the different forcings would lead to pressure gradients
differing exactly by $\nabla \phi$.
Thus, the {\em  notion of velocity equivalence} of two functionals in $H^{-1}(\Omega)^d$ can be formally defined by
\begin{equation}
    \bff_1 \simeq \bff_2 \quad :\Leftrightarrow \quad \exists q \in L^2_0(\Omega) :
     \forall\, \bfw \in H^1_0(\Omega)^d\;  \bl \bff_1 - \bff_2, \bfw \br_\Omega^{}  = -(q, \nabla \cdot \bfw)_\Omega^{}.
\end{equation}
The corresponding seminorm, which induces these equivalence classes of
functionals is naturally given for $\bff \in H^{-1}(\Omega)^d$ by
\begin{equation} \label{eq:semi:norm}
  \| \bff \|_{\bcalV(\Omega)'} := \sup_{\bfzero \not= \bfv \in \bcalV(\Omega)} \frac{| \bl \bff, \bfv \br_\Omega^{} |}
   {\| \nabla \bfv \|_{0,\Omega}^{} }.
\end{equation}
Clearly, the above supremum is a seminorm since
$\| \nabla \phi \|_{\bcalV(\Omega)'} = 0$ for all $\phi\in L^2(\Omega)$.

Turning back to the issue of constructing discrete stabilization operators
for dominant advection in Navier--Stokes flows,
we remark that also the strength of the advection
term has to be measured in the seminorm \eqref{eq:semi:norm}, and not
in the standard $H^{-1}(\Omega)^d$-norm.
Seeing things from this angle, we see that a non-zero convective term lies in between the following 
two extreme cases:
\begin{enumerate}
\item a gradient field: no dominant advection in the sense above due to
$\| (\bfu \cdot \nabla) \bfu \|_{\bcalV(\Omega)'}=0$; 
\item a divergence-free field: leading to dominant advection.
\end{enumerate}

 In the first extreme case (i.e. (1) above), where
$\| (\bfu \cdot \nabla) \bfu \|_{-1,\Omega}^{}$ is large,
although it holds $\| (\bfu \cdot \nabla) \bfu \|_{\bcalV(\Omega)'}=0$,
pressure-robust mixed methods have been shown recently to
outperform classical mixed methods that are only LBB-stable
\cite{gls:2019,fkls:ijnmf:2019}. 
They are designed in such a way that any gradient forcing
in $(\bfu_h \cdot \nabla) \bfu_h$ does not change the discrete velocity solution $\bfu_h$,
respecting on the discrete level the equivalence classes
that are induced by the seminorm \eqref{eq:semi:norm}.
From a more applied point of view, pressure-robust methods have been shown
to be important for vortex-dominated flows \cite{gls:2019,sl:sicom:2018}, where 
the following relation between the convective term and the pressure gradient holds
\begin{equation} \label{eq:quadratic:linear:balance}
    (\bfu \cdot \nabla) \bfu + \nabla p \approx \bfzero,
\end{equation}
meaning that the centrifugal force within the vortex structure
is balanced by the pressure gradient.
Such flows are known as generalized Beltrami flows and are intensively studied
in Topological Fluid Dynamics, cf. \cite{ak:book:1998}, and they
are popular benchmark problems.
For this type of flows, due to \eqref{eq:quadratic:linear:balance} the
quadratic velocity-dependent
convection  term balances the linear pressure gradient, and then the pressure field
is usually more complicated than the velocity field. As a consequence,
it has been  demonstrated numerically  for this class of time-dependent
high Reynolds number flows
that pressure-robust DG methods of order $k$ delivered
on coarse grids similarly accurate results as DG methods of order $2 k$
that are only LBB stable \cite{gls:2019}.

With respect to (2) above, in order to derive an appropriate convection stabilization for the
divergence-free part of $(\bfu \cdot \nabla) \bfu$, which is actually measured
by the seminorm $\| (\bfu \cdot \nabla) \bfu \|_{\bcalV(\Omega)'}$, we try to obtain
a better intuition for the meaning of the weak formulation \eqref{eq:weak:nse}.
Exploiting that every divergence-free function $\bfv \in \bcalV(\Omega)$  has a vector potential $\bfv = \curl \bfchi$ \cite{gr:book:1986}, 
we can formally derive for smooth enough functionals $\bff$
$$
  (\bff, \bfv)_\Omega^{} = (\bff, \curl \bfchi)_\Omega^{} = (\curl \bff, \bfchi)_\Omega^{}.
$$
When applied to the term $\bfu_t$, a similar
integration by parts with the curl operator
and introducing the vorticity $\bfomega := \curl \bfu$ 
will yield
$$
  \frac{\mathrm{d}}{\mathrm{dt}} (\bfu(t), \bfv)_\Omega^{} = \frac{d}{\mathrm{dt}} (\bfu(t), \curl \bfchi)_\Omega^{} = \frac{d}{\mathrm{dt}} (\bfomega(t), \bfchi)_\Omega^{},
$$
and applying the same idea to the remaining terms in \eqref{eq:weak:nse}
reveals that
the weak formulation \eqref{eq:weak:nse} can be understood as
a mathematically precise formulation of the vorticity equation
\begin{equation}\label{vort:equation}
  \bfomega_t -\mu \Delta \bfomega
   + \left (\bfu \cdot \nabla \right ) \bfomega 
    - \left (\bfomega \cdot \nabla \right ) \bfu = \curl \bff\,,
\end{equation}
cf. \cite{JLMNR2017,cm:book:1993}. In this last equation, the gradient of the pressure has completely disappeared.
So, starting from this remark  in this work
we propose a residual-based stabilization of the vorticity equation,
which we call {\em least squares vorticity stabilization (LSVS)}. This stabilization strategy
includes a higher order stabilization term on the vorticity equation in the bulk,
and a penalty on the jump of the tangential component of the
convective derivative over element faces (see \S~\ref{FEM:method} for details).
A similar starting point was used in the meteorology community
\cite{nc:2017} where a residual SUPG-like method built from \eqref{vort:equation} for the two-dimensional
case (although different from the one proposed in this work, and no analysis was presented in that work). 
The same principle has  also been applied
in recent work on pressure-robust residual-based a posteriori error 
control \cite{lms:nm:2019,kanschat2014}.

To keep the technical details down
we restrict the analysis to a linearized problem, namely
the following Oseen's problem on a bounded, connected, contractible, polyhedral Lipschitz domain $\Omega$:
\begin{align}\label{eq:oseen}
\calL \bfu+ \nabla p & = \bff \qquad \textrm{in}\;\Omega\,,\nonumber\\
\dive \bfu & = 0\qquad \textrm{in}\;\Omega\,,\\
\bfu & = \bfzero\qquad\textrm{on}\; \partial \Omega, \nonumber
\end{align}
where
\begin{equation}\label{L-definition}
\calL \bfu:= \sigma \bfu + (\bfbeta \cdot \nabla ) \bfu  - \mu \Delta \bfu\,.
 \end{equation}
Here, $\mu>0$ denotes the diffusion coefficient, $\sigma > 0$, and the convective term $\bfbeta$ is assumed to
belong to $W^{1,\infty}(\Omega)^d$ and to satisfy $\dive \bfbeta = 0$.
This is an elliptic system that is well posed in
$H^1_0(\Omega)^d \cap \bcalV(\Omega) \times L^2_0(\Omega)$ by Lax--Milgram's lemma
and Brezzi's theorem for all $\mu>0$.
A weak formulation of Oseen's problem, which is in the spirit of the
weak formulation \eqref{eq:weak:nse} for the time-dependent incompressible
Navier--Stokes equations, is  given by:
find $\bfu \in \bcalV(\Omega)$ such that for all $\bfv \in \bcalV(\Omega)$ the following holds
\begin{equation} \label{eq:weak:Oseen}
  \mu (\nabla \bfu, \nabla \bfv)_\Omega^{} + ((\bfbeta \cdot \nabla ) \bfu, \bfv)_\Omega^{}
    +  \sigma (\bfu, \bfv)_\Omega^{}  = \bl \bff, \bfv \br_\Omega^{}.
\end{equation}
In the following, we will refer to this weak formulation as Oseen's problem. We will nevertheless always keep in mind that,
given the unique solution of \eqref{eq:weak:Oseen}, there exists a unique pressure $p\in L^2_0(\Omega)$
such that $(\bfu,p)$ satisfies the mixed weak formulation of \eqref{eq:oseen}.

\subsection{Preliminary results}

As it was already mentioned, for every divergence-free function in $\Omega$ we can associate a vector potential. So,
natural spaces to consider that can capture the kernel of the divergence operator are given by
\begin{equation*}
\bZ:=
\left\{ \begin{array}{lr}
\{ \bz \in H^1(\Omega)^3: \curl \bz \in H_0^1(\Omega)^3\},  & \text{ if } d=3\,, \\
\{ z \in H^1(\Omega): \curl z \in H_0^1(\Omega)^2\}, & \text{ if } d=2. 
\end{array}
\right.
\end{equation*}
We stress the fact that for $d=2$, $z$ is a scalar function, while for $d=3$, $\bz$ is a vector-valued function. To simplify the presentation from
now on we will just use the boldface notation for both cases, and the definition will depend on the context.

Using the generalized Bogovskii operator since $\Omega$ is contractible and Lipschitz there exists $\bz$ with components in $H_0^2(\Omega)$ \cite{Costa10} such that 
\begin{equation}\label{aux1}
\curl \bz=\bfu \quad \text{ in } \Omega.
\end{equation}
It is important to notice here that $\bz$ and its first derivative vanish on  $\partial \Omega$. If we assume more regularity of $\bfu$ then we can find a smoother $\bz$ satisfying \eqref{aux1}; 
however, it may not satisfy boundary conditions. More precisely, the following result is a rewriting of \cite[Theorem~4.9~b)]{Costa10}, where we have used that, since the domain
$\Omega$ is supposed to be contractible, then the cohomology space is zero.

\begin{proposition}\label{prop1}
Let $\Omega \subset \mathbb{R}^d$ be a contractible, Lipschitz polygonal/polyhedral domain.  Let $\bfu \in H^r(\Omega)^3$ with $r \ge 1$ such that $\dive\bfu=0$. Then,  there exists 
$\bz\in H^{r+1}(\Omega)^d$ satisfying \eqref{aux1} and the following stability estimate 
\begin{equation}\label{aux2}
\|\bz\|_{r+1,\Omega}  \le C \|\bfu\|_{r,\Omega}\,,
\end{equation}
where the constant $C>0$  is independent of $\bfu$.
\end{proposition}
Note that the boundary conditions $\bz=0$ on $\partial \Omega$ might not hold even if $\bfu$ vanishes on  $\partial \Omega$  if we would like $r \ge 2$. However, in two dimensions we can guarantee that  boundary conditions are satisfied.  
 
\begin{corollary}\label{cor1}
Under the hypotheses of Proposition~\ref{prop1}, if  $d=2$ we can choose $\bz$ satisfying \eqref{aux1} and \eqref{aux2}, so that $\bz=0$ on $\partial \Omega$. 
\end{corollary}

\begin{proof}
Let us assume that $d=2$. By  Proposition \ref{prop1} there exists $\bz \in H^{r+1}(\Omega)$ so that \eqref{aux1} and \eqref{aux2} hold. Since $\curl \bz=\bfu$ we have that $\curl \bz=0$ on  $\partial \Omega$. Denoting by $\bft$ the unit tangent vector to $\partial\Omega$, 
this implies that $\nabla \bz \cdot \bft=0$ on  $\partial \Omega$ and then $\bz$ is 
constant on $\partial \Omega$. Let us denote that constant $c\in\mathbb{R}$. Then, the function $\tilde{\bz}=\bz-c$ satisfies all the requirements of the result, including 
 estimate \eqref{aux2}. 
\end{proof}

\section{The stabilized finite element method}
\label{sec:fem:spaces}

\subsection{Finite element spaces}

 Let $\{\calT_h^{} \}_{h>0}^{}$ be a family of shape-regular simplicial triangulations of
$\Omega$. The elements of $\calT_h^{}$ will be denoted by $K$ with diameter $h_K^{}:=\textrm{diam}(K)$ and maximal mesh width $h=\max\{h_K^{}:K\in\calT_h^{}\}$.
For an element $K\in \calT_h^{}$, we define the set $\calF_K^{}$ of its facets. The set of all facets of the triangulation $\calT_h^{}$ is denoted by $\calF$ and $\calF^i$ 
denotes  the interior facets . For $F\in \calF$
we will denote $h_F^{}=\textrm{diam}(F)$, and $|F|$ the $(d-1)$-dimensional measure of $F$ (area for $d=3$ and length for $d=2$). 
The $L^2(F)$-inner product is denoted by $\bl\cdot,\cdot\br_F^{}$.
For a vector valued function $\bfv$ we define the the tangential jumps across $F= K_1  \cap K_2$ with $K_1, K_2 \in \calT_h$  as
\begin{equation*}
\jump{\bfv \times \bfn}|_F := \bfv_1 \times \bfn_1 + \bfv_2 \times \bfn_2,
\end{equation*}
where $\bfv_i=\bfv|_{K_i}$ and $\bfn_i$ is the unit normal pointing out of $K_i$. If $F$ is a boundary face then we define
\begin{equation*}
\jump{\bfv \times \bfn}|_F^{} := \bfv \times \bfn.
\end{equation*}

In addition, we introduce the following broken inner products (assuming the functions involved are regular enough so every quantity is finite):
\begin{equation}\label{broken-pi}
(v,w)_h^{}:=\sum_{K\in\calT_h^{}}(v,w)_K^{}, \quad  \bl v,w \br_{\calF^i}:=\sum_{F\in\calF^i } \bl v,w \br_F\,  \quad   \textrm{and} \quad \bl v,w \br_\calF:=\sum_{F\in\calF} \bl v,w \br_F\,,
\end{equation}
with associated norms $\|\cdot\|_h^{}$, $\|\cdot\|_{h,\calF}^{}$,  $\|\cdot\|_{h,\calF^i}$, respectively.

For $s\ge 1$ we  define the standard  piecewise polynomial Lagrange space by
\begin{equation}\label{pol-space}
\bfW_h^s:=\{ \bfw \in H_0^1(\Omega)^d: \bfw|_K^{} \in \pol_s(K)^d \quad \forall K \in \ \calT_h\}.
\end{equation}

Over $\calT_h^{}$, and for $k\ge 1$, we assume we have  finite element spaces $\bfV_h \subset H_0^1(\Omega)^d,Q_h \subset L^2_0(\Omega) $
and the associated subspace of (exactly) divergence-free functions
\begin{equation}\label{div-free-space-h}
\bcalV_h := \{\bfv_h \in \bfV_h: \mbox{ such
  that } \dive \bfv_h= 0 \mbox{ in } \Omega\}\,,
\end{equation}
satisfying the following assumptions:

\noindent (A1) $\dive \bfV_h \subset Q_h$;

\noindent  (A2) the pair $(\bfV_h , Q_h)$ is inf-sup stable;

\noindent (A3)   $\bfW_h^k \subset \bfV_h \subset \bfW_h^{r}$  for some $r, k \ge 1$;

\noindent  (A4) there exists a finite element space $\bZ_h \subset \bZ$ such that $ \curl \bZ_h=\bcalV_h$;
 
\noindent (A5) any $\bz \in   \bZ $ with components in $H^{k+2}(\Omega)$ satisfies the following estimate:
for every multi-index $\bfalpha=(\alpha_1^{},\ldots,\alpha_d^{})\in\mathbb{R}^d$, where $|\bfalpha|:=\alpha_1^{}+\ldots+\alpha_d^{}$, the following approximation holds
\begin{equation*}
\inf_{\bfpsi_h \in \bZ_h} \|h^{|\bfalpha|} \partial^{\bfalpha} (\bz -\bfpsi_h)\|_{h} \le C h^{k+2} \|\bz\|_{k+2,\Omega}^{} \quad \text{ for } |\bfalpha| \le k+1;
\end{equation*}

\noindent (A6) if $d=2$ we can choose $\bZ_h$ so that $\bZ_h \subset H_0^1(\Omega)$.  

\begin{remark}
We finish this  section by giving an alternative interpretation of (A4)-(A5). In fact, (A4)-(A5) imply that the space $\bcalV(\Omega)$ can be approximated by functions in
$\bcalV_h^{}$ in the following sense: for all $\bfv\in \bcalV(\Omega)$, the following approximation result holds
\begin{equation} \label{eq:disc:divfree}
  \inf_{\bfw_h \in \bcalV_h^{}} \| \nabla \bfv - \nabla \bfw_h \|_{0,\Omega}^{} \leq \left (1+ C_F^{} \right )
    \inf_{\bfw_h \in \bfV_h^{}} \| \nabla \bfv - \nabla \bfw_h \|_{0,\Omega}^{}.
\end{equation}
Here, $C_F^{}$ denotes the stability constant
of a Fortin operator, whose existence is assured by LBB-stability, see \cite{JLMNR2017,BoffiBrezziFortin13,gr:book:1986}.
\end{remark}

\subsubsection{Examples of finite element methods satisfying (A1)-(A6)}
 Assumptions (A1)-(A6) essentially state that the finite element spaces are piecewise polynomials (so inverse inequalities are valid), and that
the space $\bcalV(\Omega)$ can be approximated, with optimal order, by the space $\bcalV_h^{}$. In addition, they state that the space of vector potentials
associated to the space $\bcalV(\Omega)$ can also be approximated, with optimal order, by the space $\bZ_h^{}$ containing the discrete vector
potentials. This last hypothesis will be vital in the error analysis. We now present a few examples of
finite element spaces that satisfy Assumptions (A1)-(A6).
The most classical example  (and the one we use in our numerical experiments)
 is the    Scott--Vogelius element \cite{SV85}, where 
\begin{equation}\label{SV-spaces}
\bfV_h^{}=\bfW_h^k\qquad \textrm{and}\qquad Q_h^{}=\{q_h^{}\in L^2_0(\Omega): q_h^{}|_K^{}\in \mathbb{P}_{k-1}^{}(K)\;
\forall K\in\calT_h^{}\}\,.
\end{equation} 
The Scott--Vogelius element is LBB-stable on different kinds of
shape-regular triangulations for different kinds of polynomial orders.
For example, on shape-regular, barycentrically refined meshes, the condition $k \geq d$
suffices \cite{qin:phd:1994,zhang:mcom:2005,guzman2018inf}. For $d=2$, the condition $k \geq 4$
allows to derive LBB-stability on rather general, shape-regular meshes \cite{GS19,SV85}, with potentially modifying the pressure space to allow singular vertices.
Characterizing the discrete potential space $\bZ_h$ for the above examples has been  addressed in several papers \cite{falk2013stokes, fu2020exact} and they usually form an exact sequence. In particular, the space $\bZ_h$ in the case $d=2$ on  barycentrically refined meshes is the Clough--Tocher $C^1$ space \cite{CloughTocher1965}.  Additional exact sequences, possibly using even smoother spaces, that lead to spaces satisfying our assumptions can be found in \cite{guzman2020exact,christiansen2018generalized, falk2013stokes, neilan2015discrete}.

In addition, it is worth mentioning that the requirement (A3), stating that the functions used to approximate the velocity are piecewise polynomials prevents
us from using  spaces using rational functions, such as the ones proposed in  \cite{guzman2014conforming, guzman2014conforming2}. Nevertheless, the same
analysis carried out below can be applied, with minor modifications, to that case as well.
The same observation can be made about methods that belong to the IGA family proposed in, e.g., \cite{BdFSV11,eh:jcp:2013,EH13}, since they are built using smooth rational functions,
rather than polynomials.

\subsection{The method}\label{FEM:method}
The idea is to remove all the gradient fields from the momentum equations
in the stabilization, including the pressure gradient, by
adding stabilization only
on the vorticity equation instead of the velocity-pressure one,
since any gradient is in the kernel of the curl operator.
This amounts to adding a least squares term of the
vorticity equation
$\curl \mathcal{L}\bfu  = \curl \bff$.
We multiply this equation by $\tau \curl  \mathcal{L}\bfv$,
where $\tau$ is a stabilization parameter chosen so that the stabilizing term scales in the
same way as the equation (see \eqref{tau-def} below). This leads to the term
\[
(\tau \curl \mathcal{L}\bfu,\curl \mathcal{L}\bfv)_h^{} = (\tau \curl \bff, \curl \mathcal{L}\bfv)_h^{}\,,
\]
or 
\[
(\tau \curl \mathcal{L}\bfu,\curl (\bfbeta \cdot \nabla) \bfv)_h^{} = (\tau \curl \bff, \curl (\bfbeta \cdot \nabla) \bfv)_h^{}\,.
\]
For simplicity we only consider the former form for the analysis below.
Observe that this is a high order term, which for smooth flows can be
assumed to be of a smaller magnitude than the boundary penalty term introduced next. In fact, if no
further assumptions are made on the velocity space it is not sufficient
to guarantee optimal bounds. Thus, a further control on the jumps of the convective
gradients over the facets, similar to that proposed in \cite{BH04}, needs to be added to the formulation. 
So, on each internal facet $F$ we add the term
\[
\bl h^2 \jump{(\bfbeta \cdot \nabla)\bfu_h\times \bfn},\jump{(\bfbeta \cdot \nabla) \bfv_h\times \bfn}\br_{F}^{}\,.
\] 

Gathering the terms introduced above, the stabilized finite element method analyzed in this work reads: Find $(\bfu_h^{},p_h^{}) \in \bfV_h^{}\times Q_h^{}$ such that
\begin{equation}\label{eq:FEMstab}
\hspace*{-0.4cm}
\left\{
\begin{array}{rll}
a(\bfu_h^{},\bfv_h^{}) - b(p_h^{},\bfv_h^{})+S(\bfu_h^{},\bfv_h^{}) =&  L(\bfv_h^{}) \quad & \forall \, \bfv_h \in  \bfV_h\,, \\
b(q_h^{}, \bfu_h^{})=& 0  \quad & \forall \, q_h \in Q_h\,,
\end{array}
\right.
\end{equation}
where the bilinear forms are defined by
\begin{align}
a(\bfu_h^{},\bfv_h^{}) :=&  \,(\sigma \bfu_h^{} + (\bfbeta \cdot \nabla) \bfu_h^{},\bfv_h^{})_\Omega^{} + \mu(\nabla \bfu_h^{},\nabla \bfv_h^{})_\Omega^{}\,,\\
b(p_h^{},\bfv_h^{}) :=&  \,(p_h^{}, \nabla \cdot \bfv_h^{})_\Omega^{}\,,
\end{align}
and the stabilizing bilinear form is given by
\begin{equation}\label{eq:stab_pv}
S(\bfu_h^{},\bfv_h^{}):= \delta_0^{}\Big\{(\tau \curl \calL \bfu_h, \curl \calL \bfv_h)_h+ 
\bl h^2 \jump{(\bfbeta \cdot \nabla)\bfu_h^{}\times \bfn },\jump{(\bfbeta \cdot \nabla) \bfv_h^{} \times \bfn}\br_{\calF^i}^{}\Big\}\,.
\end{equation}
Here the broken scalar products are defined in \eqref{broken-pi}, the stabilization parameter $\tau|_K^{}=\tau_K^{}$ is given by
\begin{equation}\label{tau-def}
\tau_K^{}:= \min\left\{ 1, \frac{\|\bfbeta\|_{\infty,\Omega}^{}h_K^{}}{\mu}\right\}\, \frac{h_K^3}{\|\bfbeta\|_{\infty,\Omega}^{}}\,.
\end{equation}
 Finally, the right-hand side $L$ is given by
\begin{equation}\label{RHS-method}
L(\bfv_h^{}) := (\bff,\bfv_h^{})_\Omega^{}+ \delta_0^{}( \tau \curl \bff, \curl \calL \bfv_h^{})_h^{}\,.
\end{equation}
In the stabilizing terms, $\delta_0^{}>0$ is a non-dimensional parameter. The value of $\delta_0^{}$ does not affect the qualitative behavior of the error estimates,
so we will not track this constant in our error estimates below. Nevertheless, in Section~\ref{sec:numerics} we will carry out a comprehensive study of its
optimal value.

For the analysis we introduce the following mesh-dependent norm
\begin{equation}\label{norm}
\vertiii{\bfv}^2 := \|\sigma^{\frac12} \bfv\|^2_{0,\Omega} + \|\mu^{\frac12} \nabla\bfv \|^2_{0,\Omega} + |\bfv|_S^2\,,
\end{equation}
where $|\bfv|_S^2 := S(\bfv,\bfv)$. We see that
\begin{equation}\label{elipt-a+s}
\vertiii{\bfv_h^{}}^2 =  (a+S)(\bfv_h^{},\bfv_h^{})\qquad \forall\,\bfv_h^{}\in \bfV_h^{}\,.
\end{equation}
In addition,  the  pair $\bfV_h\times Q_h$ satisfies the inf-sup condition, by Assumption (A2), which ensures the well-posedness of 
Problem~\eqref{eq:FEMstab}. Moreover, Method~\eqref{eq:FEMstab} is strongly consistent for smooth enough $(\bfu,p)$, this is
\begin{equation}\label{eq:gal_ortho}
\hspace*{-0.4cm}
\left\{
\begin{array}{rll}
a(\bfu-\bfu_h^{},\bfv_h^{}) - b(p-p_h^{},\bfv_h^{})+S(\bfu-\bfu_h^{},\bfv_h^{}) =&  0 \quad & \forall \, \bfv_h \in  \bfV_h\,, \\
b(q_h^{}, \bfu-\bfu_h^{})=& 0  \quad & \forall \, q_h \in Q_h\,. 
\end{array}
\right.
\end{equation}

\begin{remark} We remark that Method~\eqref{eq:FEMstab} can be also written, equivalently, in the following compact form:
Find $\bfu_h^{}\in\bcalV_h^{}$ such that
\begin{equation}\label{compact}
a(\bfu_h^{},\bfv_h^{}) +S(\bfu_h^{},\bfv_h^{}) =  L(\bfv_h^{}) \quad  \forall \, \bfv_h \in  \bcalV_h\,.
\end{equation}
This simplified form may be chosen for the analysis, as it does not involve the discrete pressure. However,
we do prefer to write \eqref{eq:FEMstab} involving both pressure and velocity, as \eqref{compact} can not be
implemented in an easy way, due to the necessity to identify the exactly divergence-free space
$\bcalV_h^{}$, and its basis functions. This task is, in general, not straightforward.
\end{remark}

\begin{remark}
In case of classical LBB-stable methods like the Taylor--Hood, Bernardi--Raugel, 
or the mini elements,
a similar approach employing the corresponding space
of discretely divergence-free vector fields (still denoted by $\bcalV_h^{}$, but note that its elements
are no longer exactly divergence-free) would lead to:
for all $\bfv_h \in \bcalV_h^{}$ it holds
\begin{equation}
  a(\bfu - \bfu_h, \bfv_h) +S(\bfu - \bfu_h, \bfv_h) = -(\nabla p, \bfv_h)_\Omega^{},
\end{equation}
i.e., a consistency error of the form $-(\nabla p, \bfv_h)_\Omega^{}$ appears.
Introducing the notion of a discrete Helmholtz--Hodge projector
$\mathbb{P}_h$ \cite{lr:jcp:2019,LinkeMerdon16} as
the $L^2(\Omega)$-projection onto the space of discretely divergence-free
vector fields $\bcalV_h^{}$, one recognizes that this consistency error quantifies
nothing else than the strength of this discrete Helmholtz--Hodge projector.
Note that the continuous Helmholtz--Hodge projector of any gradient field
$\nabla \phi \in L^2(\Omega)^d$ is zero, i.e., it has a very similar meaning
as $\curl \nabla \phi = \bfzero$, see \cite{JLMNR2017}.
For a LBB-stable method with a discrete pressure space with elementwise
polynomials of order $k_p$, it is a classical result  that the discrete
Helmholtz--Hodge projector of any smooth gradient fields vanishes with order
$k_p+1$ in the following discrete $\bcalV_h'$-norm (that can be interpreted as
a $H^{-1}(\Omega)^d$ semi norm):
$$
  \sup_{\bfzero \not= \bfv_h^{} \in \bfV_h^{}}
    \frac{|(\nabla \phi, \bfv_h^{})_\Omega^{}|}{\| \nabla \bfv_h^{} \|_{0,\Omega}^{}}
     \leq C h^{k_p + 1} | \phi |_{k_p,\Omega}^{}.
$$
But if one estimates the strength of the discrete Helmholtz--Hodge projector
in a dual seminorm linked to $L^2(\Omega)$, one only obtains:
\begin{equation} \label{eq:disc:helm:projector:gradient:fields}
  \sup_{\bfzero \not= \bfv_h^{} \in \bfV_h^{}}
    \frac{|(\nabla \phi, \bfv_h^{})_\Omega^{}|}{\| \bfv_h^{} \|_{0,\Omega}^{}}
     \leq C h^{k_p} | \phi |_{k_p+1,\Omega}^{}\,,
\end{equation}
see \cite{lr:jcp:2019}. We conjecture that
\eqref{eq:disc:helm:projector:gradient:fields} essentially explains why
it was not possible in the past to get an improved convergence order
$h^{k+\frac{1}{2}}$ for advection stabilization of different order
 LBB-stable methods like the Taylor--Hood or Bernardi--Raugel elements. The culprit of this 
 behavior are gradient fields in the momentum balance.
 Note that the discrete Helmholtz--Hodge projector
 of any pressure-robust method vanishes for arbitrary gradient fields
 \cite{lr:jcp:2019,LinkeMerdon16}, and thanks to the link between the mini element
and equal-order $\mathbb{P}_1^{}\times \mathbb{P}_1^{}$ elements, an improved
$O(h^{k+\frac{1}{2}})$ order for the velocity can also be proven for the former under advection
stabilization.
\end{remark}

\section{Analysis of the approximation error} \label{sec:numerical:analysis}

The two following results are classical, and will be used in the proof of our error estimates. The first is the following local trace
inequality: there exists $C>0$ such that for all $K\in \calT_h^{}$, $F\in\calF_K^{}$, and all $v\in H^1(K)$,
\begin{equation}\label{local-trace}
\|v\|_{0,F}^{}\le C\Big(h_K^{-\frac{1}{2}}\|v\|_{0,K}^{}+h_K^{\frac{1}{2}}|v|_{1,K}^{}\Big)\,.
\end{equation}
We also recall the following inverse inequality: for all $\ell,s,m\in\mathbb{N}$ such that $0\le \ell\le s\le m$ and all $q\in \mathbb{P}_m^{}(K)$ there exists $C>0$ such that
\begin{equation}\label{inverse}
|q|_{s,K}^{}\le Ch_K^{\ell-s}|q|_{\ell,K}^{}\,.
\end{equation}

Finally, as our main interest is to track the dependency of the error estimates  on the viscosity $\mu$, in order to avoid unnecessary technicalities, 
we will not track their dependency  on $\bfbeta$, or $\sigma$.

\subsection{An error estimate for the velocity}

In order to state the error estimates we define the following norm, for functions that are regular enough,
\begin{equation}\label{fancy-norm}
\| \bz\|_{\star}^2 :=\vertiii{\curl \bz}^2 + (h+\mu)\sum_{s=0}^4h^{2s-4}\|D^s\bz\|_h^2\,.
\end{equation}
Here, by $D^s\bz$ we mean the tensor $(\partial^{\bfalpha}\bz)_{|\bfalpha|=s}^{}$, this is, gradient for $s=1$, Hessian matrix for $s=2$, etc.
We start by proving a quasi-best approximation result with respect to this norm.

\begin{theorem}\label{thm:error_global}
Let $\bfu \in  H_0^1(\Omega)^d\cap H^{3}(\Omega)^d$ be the solution to \eqref{eq:oseen} and let $\bz$ be its corresponding potential given by Corollary \ref{cor1}. Let $(\bfu_h^{},p_h^{})$ be the solution of \eqref{eq:FEMstab}. If $d=3$ we assume, in addition,  that $\bfbeta\cdot\bfn=0$ on $\partial \Omega$.  Then, the following error estimate holds
\begin{equation}\label{error-est-1}
\vertiii{\bfu-\bfu_h^{}} \le  C \|\bz-\bfpsi_h\|_{\star} \qquad  \text{for all } \bfpsi_h \in \bZ_h. 
\end{equation}
The constant $C>0$ is independent of $h$ and $\mu$.
\end{theorem}

\begin{proof}
Let $\bfe = \bfu-\bfu_h^{}$. We let $\bfpsi_h \in \bZ_h$ be arbitrary and set $\bfw_h:=\curl \bfpsi_h$. We note that $\bfw_h \in \bcalV_h$ and then, using the Galerkin orthogonality \eqref{eq:gal_ortho} we have 
\begin{equation}\label{the-first}
\vertiii{\bfe}^2= a(\bfe,\bfu-\bfw_h^{})+S(\bfe, \bfu-\bfw_h^{}). 
\end{equation}
We bound the right-hand side of \eqref{the-first} term by term. For the rest of the proof, $\epsilon>0$ is arbitrary  but will be chosen sufficiently small later. 
Using Cauchy--Schwarz's and Young's inequalities we see that
\begin{equation}\label{error-proof-3}
S(\bfe, \bfu - \bfw_h^{}) \le \epsilon \vertiii{\bfe}^2 + C \vertiii{\bfu - \bfw_h}^2.
\end{equation}
We re-write the first term in \eqref{the-first} by  adding  and subtracting (element-wise) $\mu\Delta\bfe$ to obtain
\begin{equation}\label{error-proof-4}
a(\bfe, \bfu - \bfw_h) = (\calL \bfe,\bfu - \bfw_h^{})_h^{} + (\mu\Delta\bfe,\bfu - \bfw_h)_h^{}+(\mu \nabla\bfe,\nabla (\bfu - \bfw_h))_\Omega^{}\,.
\end{equation}
To bound the third  term on the right-hand side of \eqref{error-proof-4}, we  proceed as in \eqref{error-proof-3} we get 
\begin{equation*}
(\mu \nabla\bfe,\nabla (\bfu - \bfw_h))_\Omega \le  \epsilon\vertiii{\bfe}^2+ C \vertiii{\bfu - \bfw_h}^2.
\end{equation*}
For the second term in \eqref{error-proof-4}
we add and subtract $\bfw_h$, use an inverse inequality and Young's inequalities, and arrive at
\begin{alignat*}{1}
 &(\mu\Delta\bfe,\bfu - \bfw_h)_h= (\mu \Delta(\bfu-\bfw_h), \bfu - \bfw_h)_h+ (\mu \Delta(\bfw_h-\bfu_h), \bfu - \bfw_h)_h \\
   \le & \, \frac{1}{2} \|h \sqrt{\mu} \Delta (\bfu - \bfw_h)\|_h^2+ \frac{1}{2}\left(1+\frac{1}{\epsilon}\right)\|h^{-1}\sqrt{\mu}(\bfu - \bfw_h)\|_h^2 + \frac{\epsilon}{2} \|h \sqrt{\mu} \Delta (\bfw_h - \bfu_h)\|_h^2 \\
  \le  & \, \frac{1}{2} \|h \sqrt{\mu} \Delta (\bfu - \bfw_h)\|_h^2+ \frac{1}{2}\left(1+\frac{1}{\epsilon}\right) \|h^{-1}\sqrt{\mu}(\bfu - \bfw_h)\|_h^2 + C\,\frac{\epsilon }{2} \| \sqrt{\mu}\, \nabla(\bfw_h - \bfu_h)\|_h^2 \\
\le  & \, C\mu\Big( \|h  \Delta (\bfu - \bfw_h)\|_h^2+ \|\nabla(\bfu-\bfw_h^{})\|^2_{0,\Omega}+\|h^{-1}(\bfu - \bfw_h)\|_h^2\Big) + C\,\frac{\epsilon }{2} \| \sqrt{\mu}\, \nabla(\bfu - \bfu_h)\|_h^2 \\
\le &\,   C \mu\sum_{s=1}^3h^{2s-4}\|D^{s}(\bz - \bfpsi_h)\|^2_h + C \epsilon\,  \| \sqrt{\mu}\, \nabla \bfe\|_h^2 \\
\le & \,  C \|\bz - \bfpsi_h\|_{\star}^2 + C \epsilon\,\vertiii{\bfe}^2.
\end{alignat*}

We are only left with the bound for the first term on the right-hand side of \eqref{error-proof-4}. 
 First, integrating by parts we rewrite it as follows
\begin{alignat}{1}
 (\calL \bfe,\bfu - \bfw_h)_h =&  (\calL\bfe,\curl(\bz-\bfpsi_h))_h=(\curl \calL\bfe, \bz-\bfpsi_h)_h +  \bl \jump{\calL  \bfe \times \bfn}, \bz-\bfpsi_h \br_{\calF}. \label{eq:res_stab}
\end{alignat}
Applying the Cauchy--Schwarz's and Young's inequalities leads to the following bound for the first term in the right-hand side of \eqref{eq:res_stab} 
\begin{equation}
(\curl \calL\bfe, \bz-\bfpsi_h)_h \le  \,\epsilon \|\tau^{\frac{1}{2}}\curl \calL\bfe\|_h^2+C  \|\tau^{-\frac{1}{2}} (\bz-\bfpsi_h)\|_h^2 \\
 \le  \,\epsilon \vertiii{\bfe}^2 + C  \|\bz-\bfpsi_h\|_{\star}^2\,,
\end{equation}
where in the last step we used that $\|\tau^{-\frac{1}{2}} (\bz-\bfpsi_h)\|_h^2\le C\,\|\bz-\bfpsi_h\|_\star^2$, independently of the value of $\mu$.
Next, for $d=2$ we use that $\bz-\bfpsi_h=\bfzero $ (that follows from  Assumption (A6)). In the case $d=3$ we decompose 
$\bfbeta = \bfbeta\cdot\bfn\,\bfn + (\bfbeta-\bfbeta\cdot\bfn\,\bfn)=:\bfbeta_n^{}+\bfbeta_t^{}$. Since $\bfbeta_t^{}$ is parallel to the
boundary $\partial\Omega$, we have that $\bfbeta_t^{}\cdot\nabla\bfe = 0$ (since $\bfe=\bfzero$ on $\partial\Omega$). So, using 
$\bfe=\bfzero$  and $\bfbeta_n^{}=0$ (if $d=3$) on $\partial\Omega$ we see that the second term is equal to
\begin{align}
 \bl \jump{\calL \bfe \times \bfn}, \bz-\bfpsi_h \br_{\calF} &=   \bl \jump{(\bfbeta \cdot \nabla) \bfe \times \bfn}, \bz-\bfpsi_h \br_{\calF^i}
+\underbrace{\bl (\bfbeta \cdot \nabla) \bfe \times \bfn, \bz-\bfpsi_h \br_{\partial\Omega}^{}}_{=0} \nonumber\\
&\qquad +\bl \jump{ -\mu \Delta \bfe \times \bfn}, \bz-\bfpsi_h \br_{\calF} \nonumber\\
& =\bl \jump{(\bfbeta \cdot \nabla) \bfe \times \bfn}, \bz-\bfpsi_h \br_{\calF^i}
+\bl \jump{ -\mu \Delta \bfe \times \bfn}, \bz-\bfpsi_h \br_{\calF} \,. \label{MM}
\end{align}
To bound the first term we use Young's inquality and the local trace theorem \eqref{local-trace} to get to
\begin{align*}
  \bl \jump{\bfbeta \cdot \nabla \bfe \times \bfn}, \bz-\bfpsi_h \br_{\calF^i} \le & \,\epsilon \| h \jump{(\bfbeta \cdot \nabla) \bfe \times \bfn}\|_{h,\calF^i}^2+ C \| h^{-1} (\bz-\bfpsi_h)\|_{h,\calF^i}^2  \\
\le & \, C\,\epsilon \vertiii{\bfe}^2 +   C\,\sum_{s=0}^1 h^{2s-3}\|D^s(\bz-\bfpsi_h)\|^2_h\\
   \le & \,C\,\epsilon \vertiii{\bfe}^2 + C  \|\bz-\bfpsi_h\|_{\star}^2.
\end{align*}
For the remaining term in \eqref{MM} we add and subtract $\bfw_h$ and get
\begin{alignat}{1}
\bl \jump{ -\mu \Delta \bfe\times \bfn}, \bz-\bfpsi_h \br_{\calF}=& \bl \jump{ -\mu \Delta (\bfu-\bfw_h)\times \bfn}, \bz-\bfpsi_h \br_{\calF}   + \bl \jump{ -\mu \Delta (\bfw_h-\bfu_h) \times \bfn}, \bz-\bfpsi_h \br_{\calF}. \label{713}
\end{alignat}
To bound the first term of \eqref{713} we apply Cauchy--Schwarz's and Young's inequalities, and the local trace result \eqref{local-trace} to arrive at
\begin{align*}
&\bl \jump{ -\mu \Delta (\bfu-\bfw_h)\times \bfn}, \bz-\bfpsi_h \br_{\calF}\le \frac{1}{2}  \|h^{3/2} \sqrt{\mu} \jump{\Delta (\bfu-\bfw_h)\times \bfn} \|_{h,\calF}^2  + \frac{1}{2}\| h^{-3/2} \sqrt{\mu}(\bz-\bfpsi_h)\|_{h,\calF}^2\\
\le& C\,\mu\, \Big( h^2\|\Delta(\bfu-\bfw_h^{})\|^2_h+h^4\|\nabla\Delta (\bfu-\bfw_h^{})\|^2_h + h^{-4}\|\bz-\bfpsi_h^{}\|^2_h+h^{-2}\|\nabla (\bz-\bfpsi_h^{})\|^2_h\Big)\\
\le& C\mu\,\sum_{s=0}^4 h^{2s-4}\|D^s(\bz-\bfpsi_h^{})\|^2_h\\
\le& C\,\|\bz-\bfpsi_h^{}\|_\star^2\,.
\end{align*}
For the second term on \eqref{713}  we use  Cauchy-Schwarz's inequality, the local trace result \eqref{local-trace}, the  inverse estimate \eqref{inverse}, and Young's inequality, leading to
\begin{alignat*}{1}
 \bl \jump{ -\mu \Delta (\bfw_h-\bfu_h)\times \bfn}, \bz-\bfpsi_h \br_{\calF} \le &\, \epsilon \vertiii{\bfw_h-\bfu_h}^2 +C \| h^{-3/2} \sqrt{\mu}(\bz-\bfpsi_h)\|_{h,\calF}^2 \\
  \le &\, 2\epsilon \vertiii{\bfe}^2 +  C  \|\bz-\bfpsi_h\|_{\star}^2\,.
\end{alignat*}
Hence, combining the above results and inserting the bounds into \eqref{the-first} gives
\begin{alignat*}{1}
\vertiii{\bfe}^2 \le C \epsilon \vertiii{\bfe}^2 +  C  \|\bz-\bfpsi_h\|_{\star}^2.
\end{alignat*}
Taking $\epsilon$ sufficiently small and re-arranging terms finishes the proof. 
\end{proof}

The last result stresses the fact that the approximation of the solution depends only on how well the space $\bZ_h$ approximates the space $\bZ$, or, in other words,
on how well the potential $\bz$ is approximated by $\bZ_h$. To make this bound more precise, we use Assumption (A5) and Corollary~\ref{cor1} to obtain the following result.

\begin{corollary}\label{bound-error}
Let us assume, in addition to the hypotheses of Theorem~\ref{thm:error_global}, that $\bfu \in  H_0^1(\Omega)^d\cap H^{k+1}(\Omega)^d$. Then, there
exists a constant $C>0$, independent of $h$ and $\mu$, such that
\begin{equation}\label{final-bound}
\vertiii{\bfu-\bfu_h^{}} \le  Ch^{k}\big(h^{\frac{1}{2}}+\mu^{\frac{1}{2}}\big)\|\bfu\|_{k+1,\Omega}^{}\,.
\end{equation}
\end{corollary}

Two conclusions can be drawn from this last result. First, that Method~\eqref{eq:FEMstab} has optimal, pressure-robust convergence rates. In addition, if the extra hypothesis
$\mu\le Ch$ is imposed, then \eqref{final-bound} leads to an $O(h^{k+\frac{1}{2}})$ error estimate. This sort of estimate has only been obtained very recently for an incompressible
problem using RT and BDM spaces in \cite{BBG20}, and, up to our best knowledge, the present result constitutes the first time such an estimate is obtained for
stabilized methods for the Oseen equation. We stress that the shape of the stabilization used is essential to obtain these results.

\color{black}


\subsection{An error estimate for the pressure}
For regular enough solutions (at least $H^3(\Omega)^d$ for the velocity), we will now show a 
superclosedeness result for the discrete pressure that depends on the velocity error estimate
only, which makes it pressure-robust. We denote by $\pi_h^{}:L^2(\Omega)\to Q_h^{}$ the $L^2(\Omega)$ orthogonal
projection onto $Q_h^{}$. 

Thanks to the Galerkin orthogonality \eqref{eq:gal_ortho} and the fact that  $\dive \bfV_h^{}\subseteq Q_h^{}$ (see (A1)) we get, for an arbitrary $\bfv_h^{}\in\bfV_h^{}$,
\begin{align*}
  a(\bfu-\bfu_h, \bfv_h) + S(\bfu-\bfu_h, \bfv_h)
   & = (p - p_h, \nabla \cdot \bfv_h)_\Omega^{} \\
   & = (\pi_h p - p_h, \nabla \cdot \bfv_h)_\Omega^{}.
\end{align*}
The $\bfV_h^{}\times Q_h^{}$ is an inf-sup stable pair (see (A2)), this guarantees the existence of a Fortin operator onto $Q_h^{}$ that commutes with the
divergence. Since, in addition $\dive \bfV_h^{} \subseteq Q_h^{}$ (see (A1)), then this operator is surjective. So, there  exists
a $\bfx_h^{}  \in \bfV_h^{}$ such that
\begin{equation}\label{inf-sup}
\dive \bfx_h^{}  = \pi_h^{}  p - p_h^{} \quad\textrm{ in }\Omega\qquad\textrm{and}\qquad \| \nabla \bfx_h^{}  \|_{0,\Omega}^{} \le C \| \pi_h^{}  p - p_h^{}  \|_{0,\Omega}^{}\,,
\end{equation}
where $C>0$ only depends on $\Omega$. 
Thus, integrating by parts, using that $\dive\bfbeta=0$, and  Cauchy--Schwarz's 
inequality, we arrive at
\begin{align}
  \| \pi_h^{} p - p_h^{} \|^2_{0,\Omega} & =
  a(\bfu-\bfu_h^{} , \bfx_h^{} ) + S(\bfu - \bfu_h^{} , \bfx_h^{} ) \nonumber\\
   & \leq  \vertiii{\bfu - \bfu_h^{} } \cdot \vertiii{\bfx_h^{} }    -((\bfbeta \cdot \nabla)\bfx_h^{} , \bfu-\bfu_h^{} )_\Omega^{} \nonumber\\
  & \leq  \vertiii{\bfu - \bfu_h^{} } \cdot \vertiii{\bfx_h^{} }
    + \|\bfbeta\|_{\infty,\Omega}^{} \| \bfu-\bfu_h^{} \|_{0,\Omega}^{}     \|\nabla \bfx_h^{}  \|_{0,\Omega}^{}\,. \label{almost-there}
\end{align}
Thanks to the stability result in \eqref{inf-sup}, once the bound $\vertiii{\bfx_h^{}}\le C\,\|\pi_h^{}p-p_h^{}\|_{0,\Omega}^{}$ is established, then 
\eqref{almost-there} provides an error estimate for $\pi_h^{}p-p_h^{}$ in terms of the velocity error estimate only. So, it only remains to bound
the triple norm of $\bfx_h^{}$. 
First, using the stability bound given in \eqref{inf-sup}  and the Poincar\'e inequality we get
\begin{align}
   \vertiii{\bfx_h^{}}
    & \leq \sigma^{\frac{1}{2}} \| \bfx_h^{} \|_{0,\Omega}^{}
     + \mu^{\frac{1}{2}}  \| \nabla \bfx_h^{} \|_{0,\Omega}^{}
     + |\bfx_h^{}|_S^{} \nonumber\\
     & \le C (\sigma^{\frac{1}{2}} + \mu^{\frac{1}{2}}) 
       \| \pi_h^{} p - p_h^{} \|_{0,\Omega}^{}
       +  |\bfx_h^{}|_S^{}\,. \label{last-bound}
\end{align}
Finally, using the inverse inequality \eqref{inverse}, the local trace result \eqref{local-trace}, and the definition of the $|\cdot|_S^{}$-seminorm, we get
\begin{equation}\label{bound-S}
  |\bfx_h^{}|_S^{} \leq C \left (  1+h^{\frac{1}{2}}+\mu\tau^{\frac{1}{2}} h^{-2}\right )    \| \pi_h p - p_h \|_{0,\Omega}^{}\,,
\end{equation}
where the constant $C$ depends on $\sigma$ and different norms of $\bfbeta$, but not on $\mu$.
Inserting \eqref{last-bound} and \eqref{bound-S} into \eqref{almost-there}, and using that $\mu\tau^{\frac{1}{2}} h^{-2}\le C\mu^{\frac{1}{2}}$, {\it regardless the value of $\mu$}, we have proven the following error estimate for the discrete pressure.

\begin{theorem} \label{thm:discrete:pressure:supercloseness}
Let us assume the hypotheses of Theorem~\ref{thm:error_global}. Then, there exists $C>0$, independent of $h$ and $\mu$, such that
\begin{equation}\label{error-p}
\| \pi_h^{} p-p_h^{}\|_{0,\Omega}^{}  \le
C  \left ( 1+\mu^{\frac{1}{2}} +h^{\frac{1}{2}}\right)\vertiii{\bfu - \bfu_h^{}}\,.
\end{equation}
\end{theorem}

\begin{remark}
The last result states that the difference $\pi_h^{}p-p_h^{}$ satisfies the same error estimate as the velocity, independently of the value
of $\mu$. In particular, this difference behaves like $O(h^{k+\frac{1}{2}})$ in the convection dominated regime.
In addition, using the triangle inequality we get 
\begin{equation}\label{opt-bound-p}
  \| p - p_h^{} \|_{0,\Omega}^{} \leq \| p - \pi_h^{} p \|_{0,\Omega}^{}
   + \| \pi_h^{} p - p_h^{} \|_{0,\Omega}^{}\,. 
\end{equation}
This, combined with the bound proven in Theorem~\ref{thm:discrete:pressure:supercloseness} and the standard approximation properties of $\pi_h^{}$ (see, e.g., \cite{gr:book:1986}),
 gives an optimal order $O(h^k)$ error estimate for the pressure whenever
$Q_h^{}$ contains piecewise polynomials of order $k-1$ (the case of, e.g., Scott--Vogelius elements of order $k$), and the pressure $p$ is regular enough.
However, due to the degree of the polynomials belonging to $Q_h^{}$, this error  bound can not be improved.
\end{remark}

\section{Numerical examples} \label{sec:numerics}
This section illustrates the theoretical findings with several numerical examples and compares the streamline-upwind Petrov--Galerkin (SUPG) method with the new least-square vorticity stabilization (LSVS) applied to the Scott--Vogelius finite element method of order $2$,  given by
\[ \bfV_h^{}=\bfW_h^2\qquad \textrm{and}\qquad Q_h^{}:= \{ q_h^{}\in L^2_0(\Omega): q_h^{}|_K^{}\in \mathbb{P}_1^{}(K)\;,\;
\forall\, K\in \calT_h^{}\}\,.\] 
 Inf-sup stability is ensured on barycentric refined triangulations as the ones used in the examples below.
The detailed implementation is stated below and all computations were performed using the finite element package ParMooN~\cite{parmoon} and are compared and confirmed
with a code written using FENiCS~\cite{Fenics}. 

The discrete problem reads: Find  $(\bfu_h^{}, p_h^{}) \in \bfV_h^{} \times Q_h^{}$ such that, for all $(\bfv_h, q_h) \in \bfV_h \times Q_h^{}$, the following holds
\begin{align}\label{eq:disc}
a(\bfu_h^{},\bfv_h^{})+b(p_h^{},\bfv_h^{})+b(q_h^{},\bfu_h^{})  +  S_\text{stab}(\bfu_h^{},\bfv_h^{})
	=   L_\text{stab}(\bfv_h^{})\,,
\end{align}
where \(S_\text{stab}\) and \(L_\text{stab}\) can be, either the novel LSVS stabilisation given by \eqref{eq:stab_pv} an \eqref{RHS-method}, 
or the SUPG stabilization given by
\begin{align*}
    S_\text{SUPG}(\bfu_h^{},\bfv_h^{}) & := \delta_0^{}\sum_{K\in \calT_h^{}} h_K^2 \left( \calL\bfu_h^{}, \bfbeta \cdot \nabla \bfv_h^{}\right)_K^{}\,,\\
    L_\text{SUPG}(\bfv_h^{}) & := (\bff,\bfv_h^{})_\Omega^{}+ \delta_0^{}\sum_{K\in \calT_h^{}} h_K^2 (\bff, \bfbeta \cdot \nabla \bfv_h^{})_K^{}.
\end{align*}
 To assess the influence of the stabilization parameter \(\delta_0\) in SUPG and LSVS methods, the positive constant \(\delta_0\) varies across the wide range from \(10^{-5}\) to \(10^{3}\). Concerning the choice of stabilization parameter for convection-dominated problems, e.g., see~\cite{AM16}, a good parameter choice for the SUPG method is $\delta_0 \in (0,1)$. 
Based on a parameter study presented in the next section, and from previous experience (see, e.g., \cite{AJ20,AM16}), all the simulations for convergence studies were performed with 
$\delta_0=0.25$ for the SUPG method and  $\delta_0^{}=0.006$ for the LSVS method. Additionally,  Example 1, Figure~\ref{fig:param_exone_supg}, confirms that
the present method presents a much more robust behavior with respect to the value of $\delta_0^{}$ than the SV-SUPG method.

\subsection{Numerical results}
We visit four different examples of the steady-state Oseen problem
defined on the domain \(\Omega=(0,1)^2\). All calculations are carried out
on non-uniform grids. Thus, a sequence of shape-regular
unstructured grids was generated, and each of these grids was barycentrically
refined, thereafter, in order to guarantee inf-sup stability.
The coarsest grid is depicted in Fig.~\ref{mesh_information}.
The corresponding velocity and pressure degrees of freedoms are listed next to it. 
\begin{figure}[ht]
    \begin{minipage}{0.5\textwidth}
    \begin{center}
		\includegraphics[width=\textwidth]{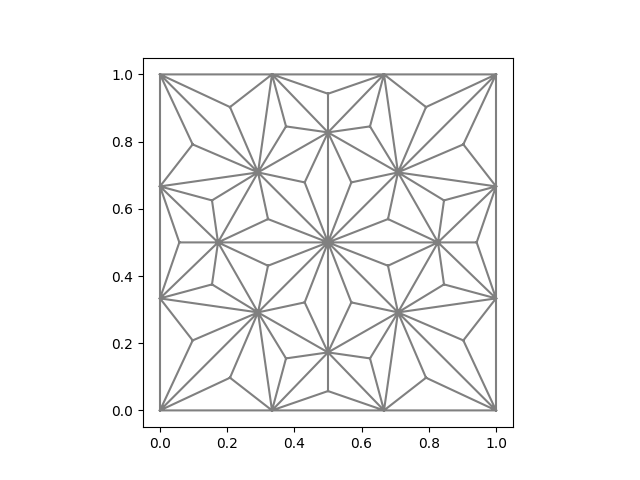}
	\end{center}
	\end{minipage}	
    \begin{minipage}{0.49\textwidth}
	\begin{center}
		\begin{tabular}{l|lll}
			level & \multicolumn{1}{c}{ndof $\bfu_h$} & \multicolumn{1}{c}{ndof $p_h$} & \multicolumn{1}{c}{total ndof}\\
			\hline
			1 &362   &252   &614 \\ 
			2 &1394  &1008  &2402\\ 
			3 &5474  &4032  &9506\\
			4 &21698 &16128 &37826\\
			5 &86402 &64512 &150914\\
		\end{tabular}	
	\end{center}
	\end{minipage}
		\caption{Initial mesh level 1 (left) and number of degrees of freedom for all refinement levels (right).}
		\label{mesh_information}
\end{figure}
In all the tables below, we use the following shorthand notation:
\begin{equation*}
L^2(u):=\|\bfu-\bfu_h^{}\|_{0,\Omega}^{}\quad,\quad H^1(u) := \|\nabla(\bfu-\bfu_h^{})\|_{0,\Omega}^{}\quad,\quad L^2(p):=\|p-p_h^{}\|_{0,\Omega}^{}\,.
\end{equation*}

\subsubsection{Example 1: Potential flow example}
The first example concerns a steady potential flow of the form \(\bfu = \nabla h\) with 
harmonic potential \(h=x^3-3xy^2\). Then, the solution 
\[
(\bfu, p) = \left( \nabla h, -\frac12 |\nabla h|^2 + \frac{14}{5} \right)\,,
\]
satisfies the Oseen problem~\eqref{eq:oseen} with the source term $\bff = \bfzero, \bfbeta = \bfu$, and inhomogeneous Dirichlet boundary conditions.

Figures~\ref{fig:nu_exone_gal}, \ref{fig:nu_exone_supg} and~\ref{fig:nu_exone_lsvs} display the results obtained by the plain divergence-free Galerkin
Scott-Vogelius finite element method (SV), the novel least-square vorticity convection stabilization (SV-LSVS) method
and the classical streamline-upwind Petrov-Galerkin (SV-SUPG) method, respctively, on refinement level 2 and the two parameter choices \(\sigma=0\) and \(\sigma=1\).

\begin{figure}[!ht]
	\centering
	\includegraphics[width=0.48\textwidth]{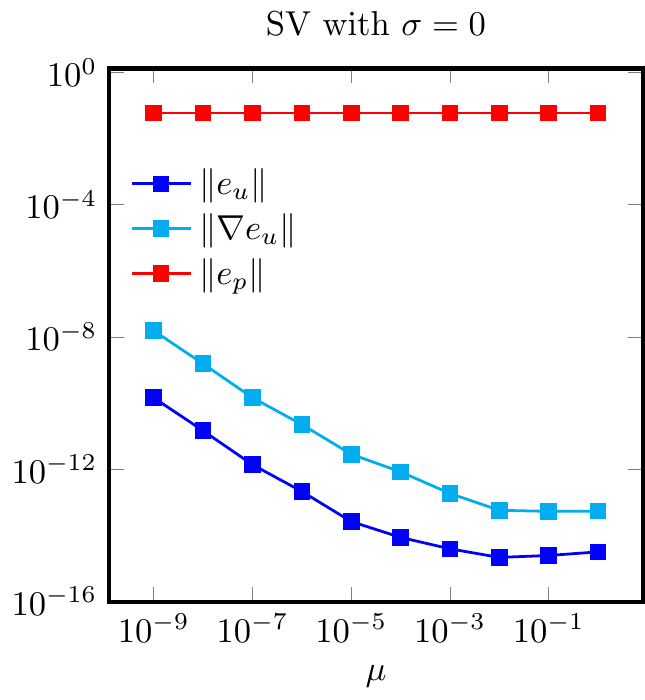}
	\includegraphics[width=0.48\textwidth]{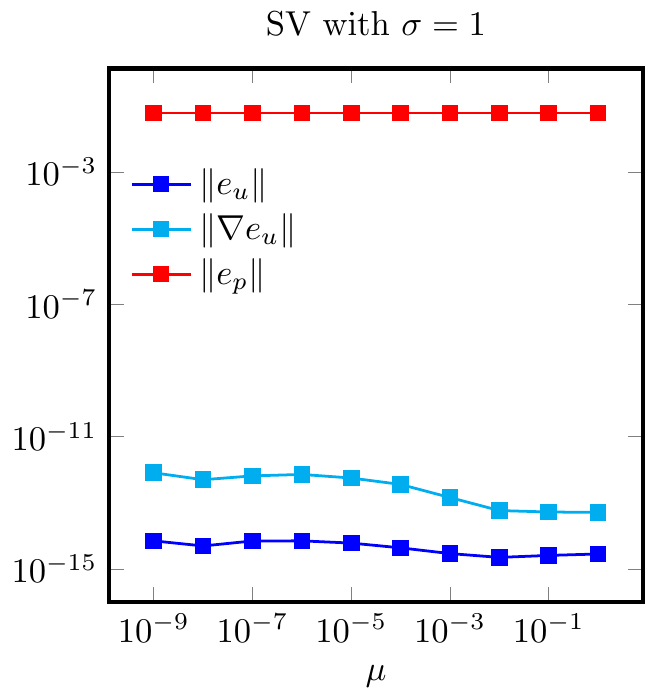}
	\caption{Example 1: error plots of different norms vs the viscosity parameter \(\mu\)
	for Scott-Vogelius finite element methods on refinement level 2 (\(\sigma=0\) left and \(\sigma=1\) right).} \label{fig:nu_exone_gal}
\end{figure}
\begin{figure}[!ht]
	\centering
	\includegraphics[width=0.48\textwidth]{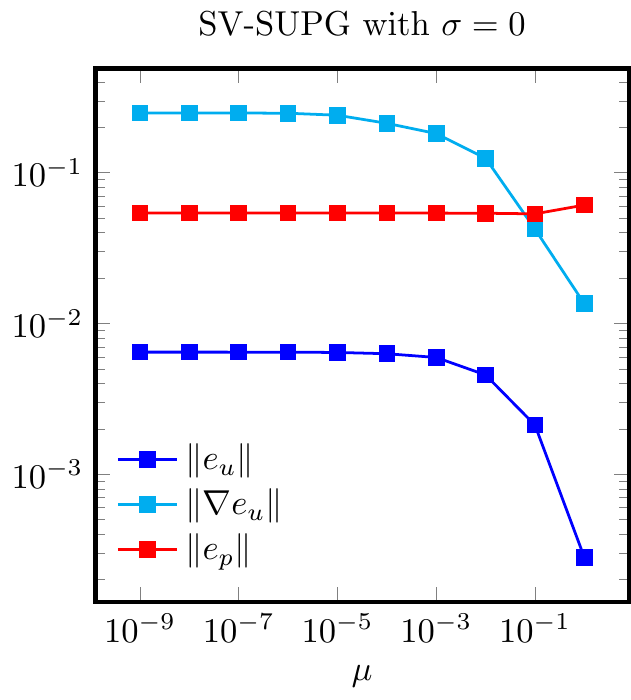}
	\includegraphics[width=0.48\textwidth]{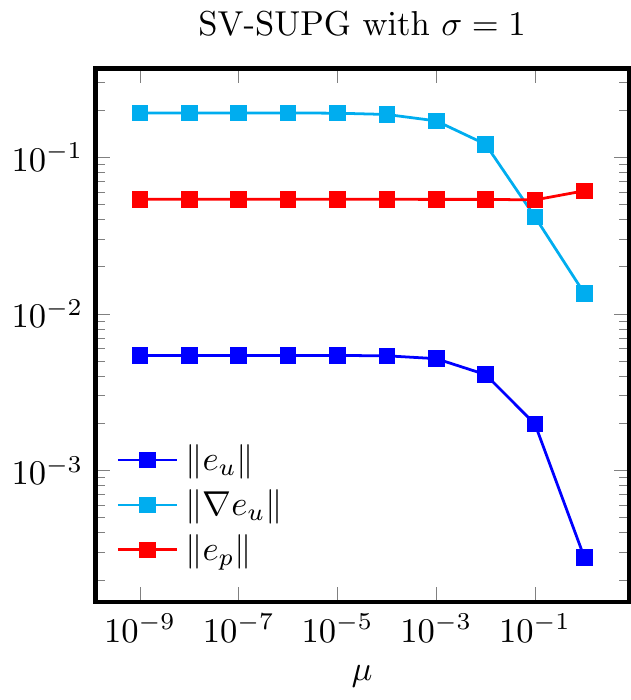}
	\caption{Example 1: error plots of different norms vs the viscosity parameter \(\mu\)
		for Scott-Vogelius element with SUPG stabilization on refinement level 2 (\(\sigma=0\) left and \(\sigma=1\) right).} \label{fig:nu_exone_supg}
\end{figure}
\begin{figure}[!ht]
	\centering
	\includegraphics[width=0.48\textwidth]{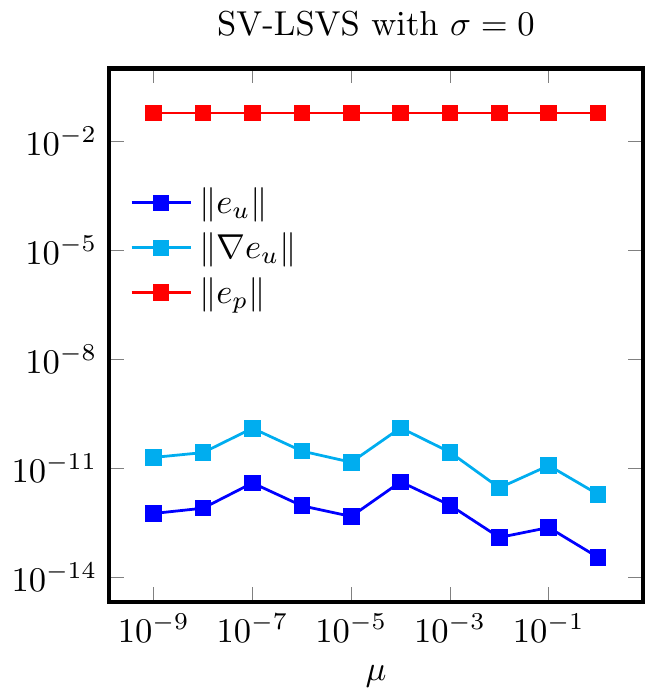}
	\includegraphics[width=0.48\textwidth]{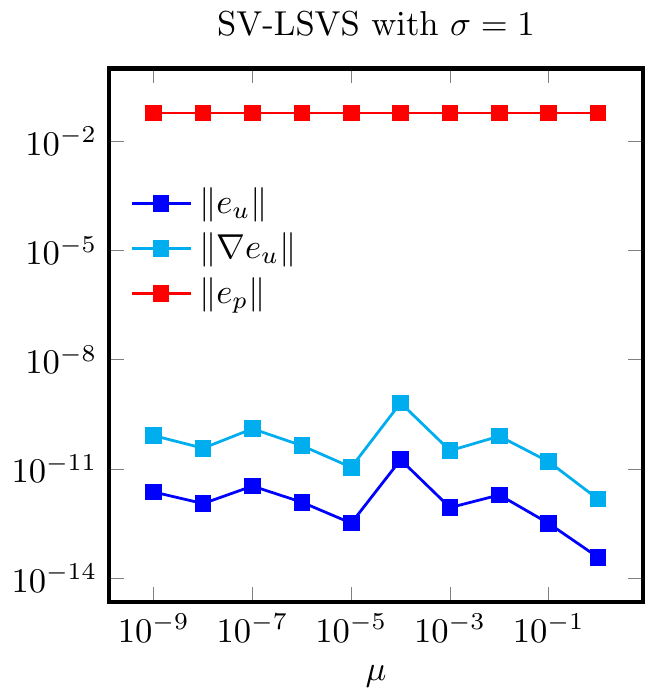}
	\caption{Example 1: error plots of different norms vs the viscosity parameter \(\mu\)
		for Scott-Vogelius finite element method with LSVS stabilization on refinement level 2 (\(\sigma=0\) left and \(\sigma=1\) right).} \label{fig:nu_exone_lsvs}
\end{figure}

\begin{figure}[!htb]
\centering
\includegraphics[width=0.48\textwidth]{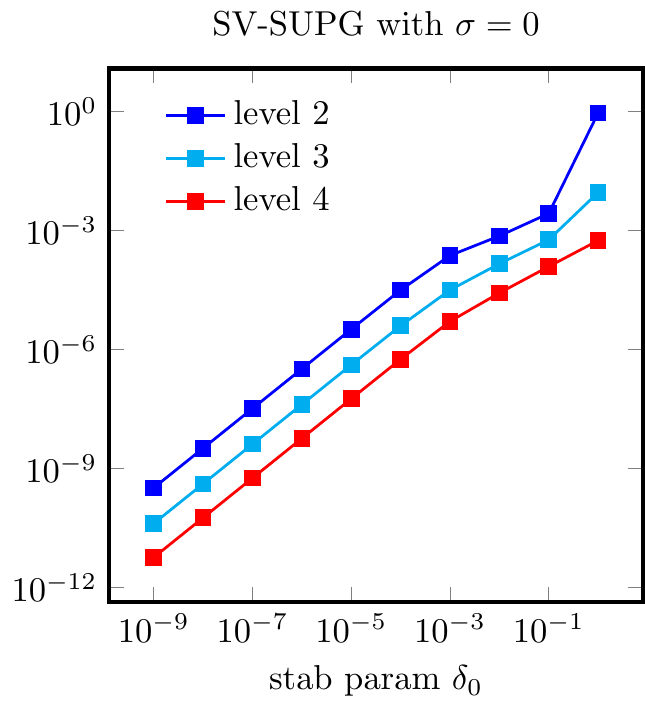}
\includegraphics[width=0.48\textwidth]{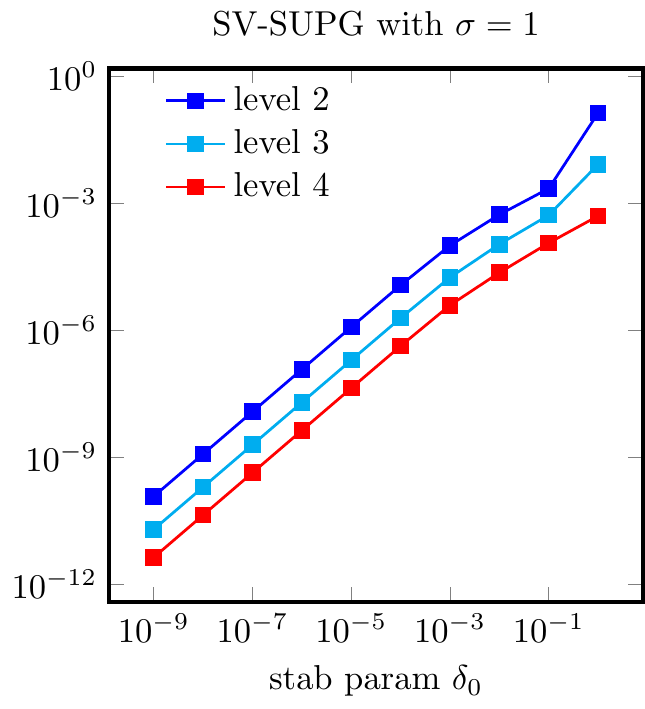}
 \caption{Example 1: $L^2$ velocity error for different stabilization parameters and different refinement levels for SV-SUPG (\(\sigma=0\)
    left and \(\sigma=1\) right) and fixed viscosity \(\mu=10^{-5}\).} \label{fig:param_exone_supg}
\end{figure}

The main observation is that both the plain SV method and the SV-LSVS method produce the exact velocity solution in this example, while the SV-SUPG method does not. Note, that this example is designed such that the exact solution belongs to the velocity ansatz space and any pressure-robust method therefore should be able to compute it exactly. Hence, this example demonstrates that SV-SUPG introduces some pressure-dependent error into the system that perturbs the discrete velocity solution. Moreover, at least in the parameter range $\mu \in [10^{-4},10^0]$ the velocity error scales with $\mu^{-1}$  which hints to a locking effect as observed for classical non-pressure-robust finite element methods in pressure-dominant situations. The effects can be explained by a closer look at the convection term. In this example $\sigma \bfu + (\bfbeta \cdot \nabla) \bfu$ completely balances the pressure gradient and therefore is a gradient itself.
A pressure-robust stabilisation does not need to stabilize gradient forces and therefore SV-LSVS (since any curl of a gradient vanishes) does not see this gradient and behaves identically to the plain SV method here
--- \textit{independent of the choice of the stabilization parameter.}
The SV-SUPG method on the other hand effectively sees and tries to stabilize the force \(\nabla_h (p - p_h)\) which does not vanish.

To round up the impression, Figure~\ref{fig:param_exone_supg} displays the $L^2$ velocity error of the SV-SUPG method on different mesh refinement levels and different choices of the SUPG stabilisation parameter \(\delta_0\). Usually, such a parameter plot leads to the conclusion that the optimal choice of \(\delta_0\) is around \(0.25\). This is not the case in this
extreme example. Here, the error scales approximately linearly with \(\delta_0\) and is optimal for \(\delta_0 = 0\), thus reinforcing the idea that the SUPG stabilization introduces a consistency error that affects the accuracy of the method.

\subsection{Example 2: Planar lattice flow}
In this example, we compare the accuracy of all methods considered in the previous example. This time the exact velocity is not in the velocity ansatz space. However, the convection term is still a gradient in the limit \(\bfu_h \rightarrow \bfu\). To this end, we fix \(\mu=10^{-5}\), \(\bfbeta = \bfu\) and boundary conditions are chosen 
such that 
\[
 \bfu = \big(\sin(2\pi x)\sin(2\pi y), \cos(2\pi x) \cos(2\pi y)\big), \qquad 
 p = \frac14 (\cos(4 \pi x) - \cos(4 \pi y))
\]
is the solution of the Oseen problem \eqref{eq:oseen} with \({\bf f} = \sigma \bfu - \mu \Delta \bfu\).
\begin{figure}[!htb]
	\centering
	\includegraphics[width=0.48\textwidth]{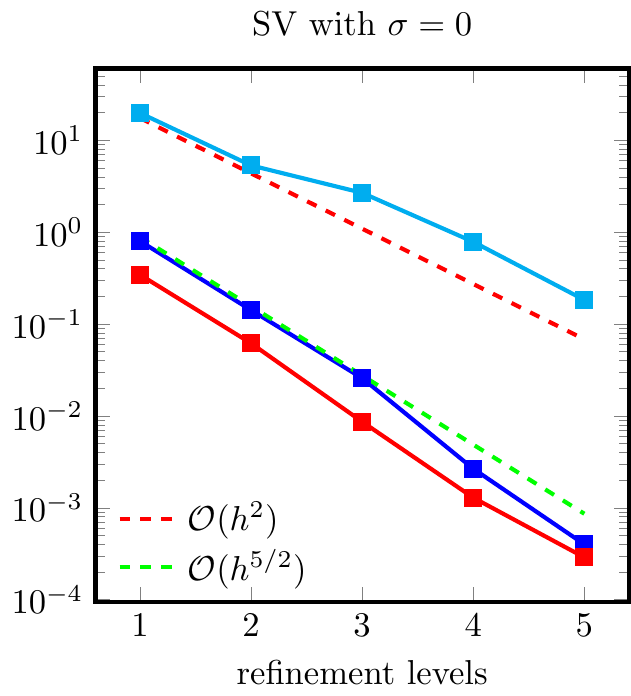}
	\includegraphics[width=0.48\textwidth]{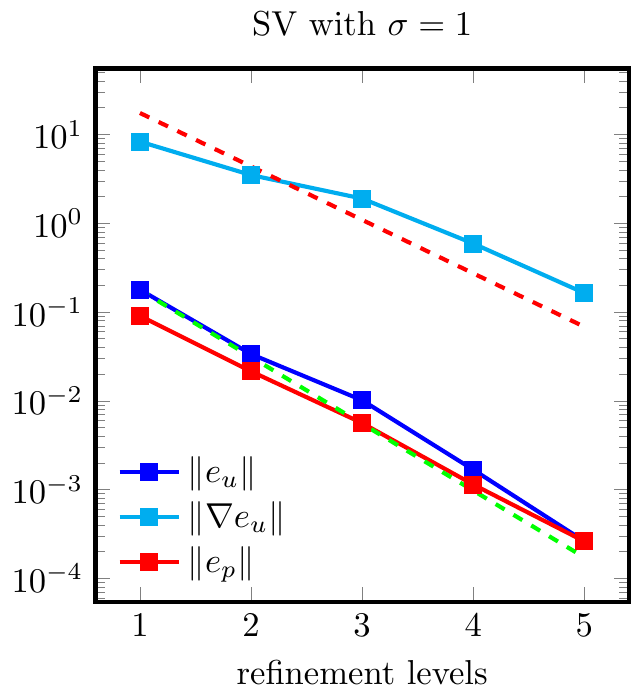}
	\caption{Example 2: error plots of different norms on different refinement levels 
	for Scott--Vogelius finite element methods (\(\sigma=0\)
	left and \(\sigma=1\) right) and fixed viscosity \(\mu=10^{-5}\).} \label{fig:nu_extwo_gal}
\end{figure}
\begin{figure}[!htb]
	\centering
	\includegraphics[width=0.48\textwidth]{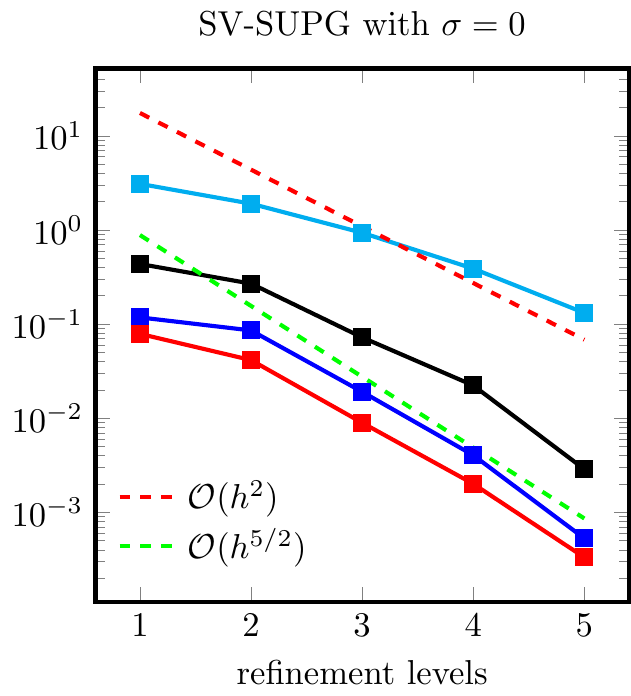}
	\includegraphics[width=0.48\textwidth]{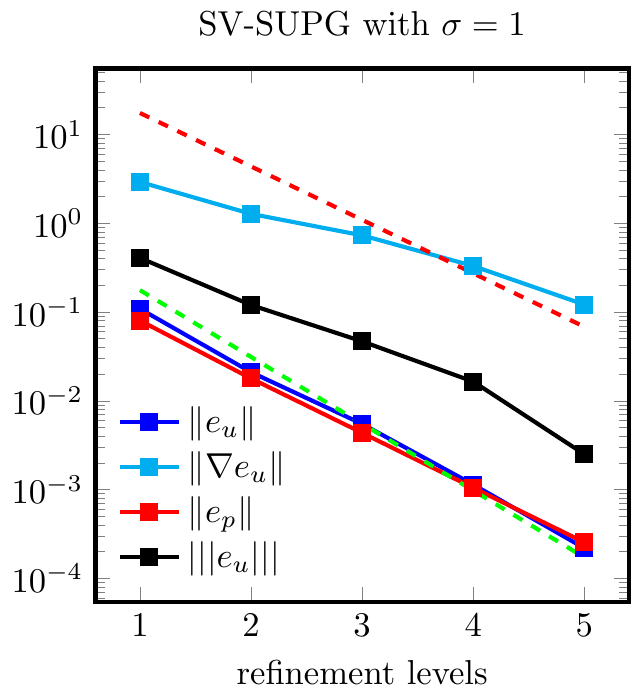}
	\caption{Example 2: error plots of different norms on different refinement levels
		for Scott--Vogelius with SUPG stabilization (\(\sigma=0\)
		 left and \(\sigma=1\) right) and fixed viscosity \(\mu=10^{-5}\).} \label{fig:nu_extwo_supg}
\end{figure}
\begin{figure}[!htb]
	\centering
	\includegraphics[width=0.48\textwidth]{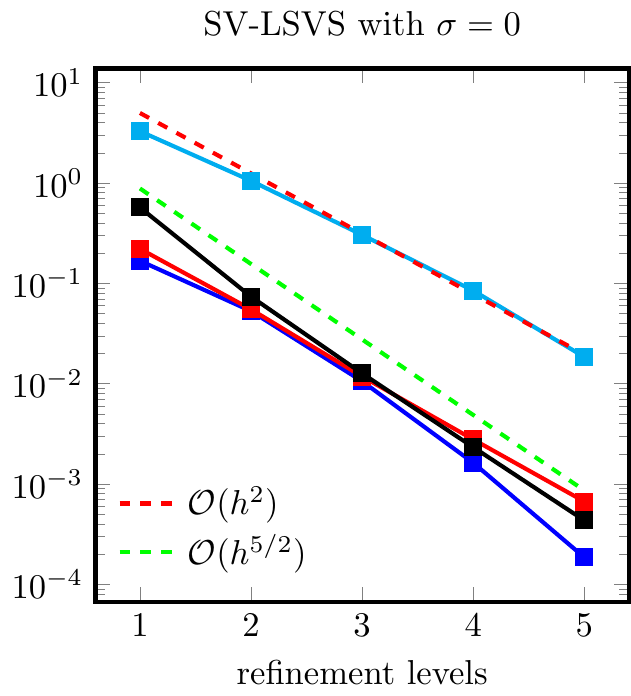}
	\includegraphics[width=0.48\textwidth]{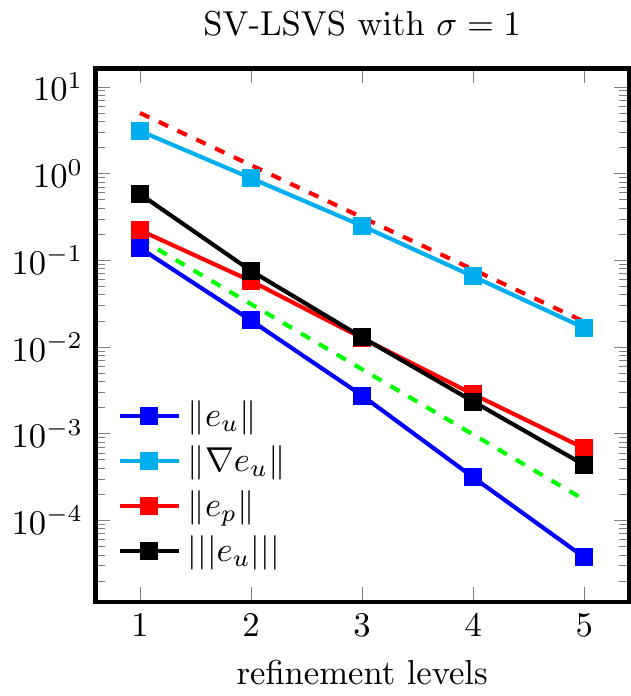}
	\caption{Example 2: error plots of different norms on different refinement levels
		for Scott--Vogelius with LSVS stabilization (\(\sigma=0\)
		left and \(\sigma=1\) right) and fixed viscosity \(\mu=10^{-5}\).} \label{fig:nu_extwo_conv}
\end{figure}

Figures~\ref{fig:nu_extwo_gal}-\ref{fig:nu_extwo_conv} display the convergence history of all three methods under consideration. The plain SV method does not convergence optimally, at least pre-asymptotically for \(\sigma = 1\)
(average EOC=$2.35$). Also the SV-SUPG method shows suboptimal behavior for \(\sigma = 1\) (average EOC=$2.24$) and for \(\sigma = 0\) (average EOC=$1.95)$. SV-SUPG is not really much more accurate than
the plain SV method on finer meshes, while it stabilizes
the solution on coarser meshes. Also, for other choices of the SUPG stabilisation parameter \(\delta_0\), see Figure~\ref{fig:param_extwo_supg}, the situation does not improve much, although the optimum on coarse meshes seems to be slightly shifted  toward smaller values.

\begin{figure}[!htb]
\centering
\includegraphics[width=0.48\textwidth]{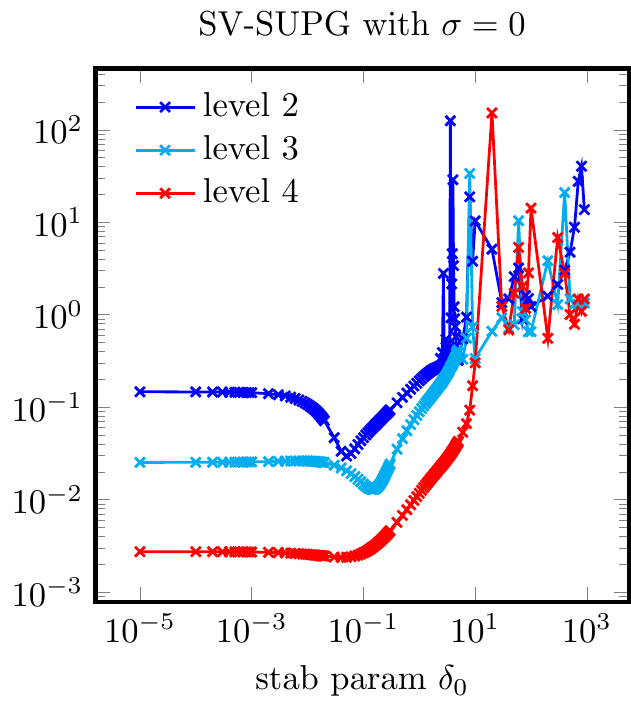}
\includegraphics[width=0.48\textwidth]{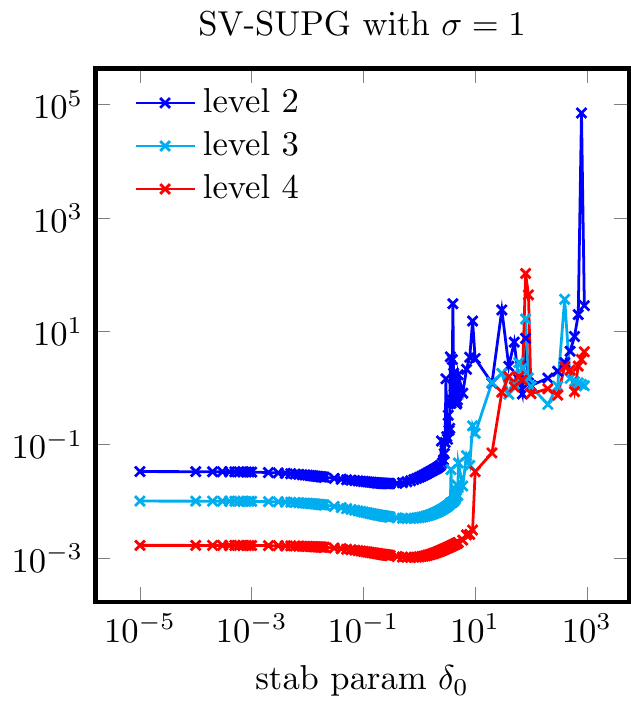}
 \caption{Example 2: $L^2$ velocity error for different stabilization parameters and different refinement levels for SV-SUPG (\(\sigma=0\)
    left and \(\sigma=1\) right) and fixed viscosity \(\mu=10^{-5}\).} \label{fig:param_extwo_supg}
\end{figure}

\begin{figure}[!htb]
\centering
\includegraphics[width=0.48\textwidth]{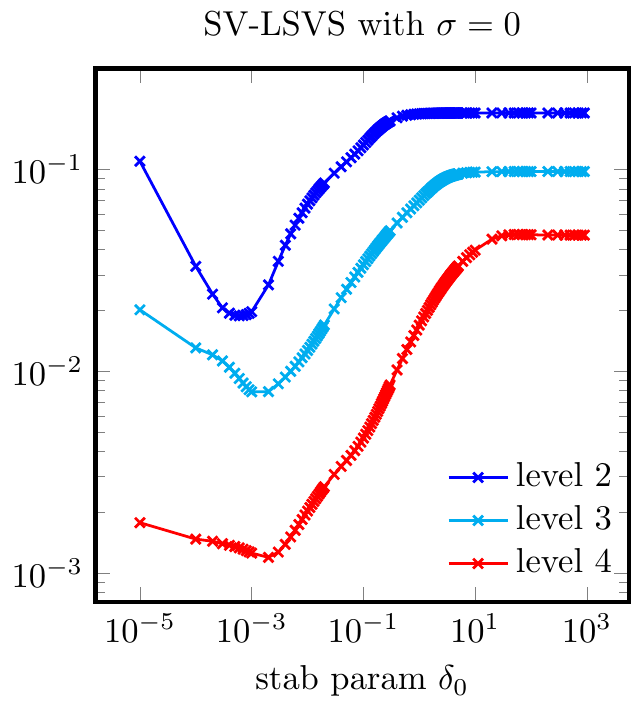}
\includegraphics[width=0.48\textwidth]{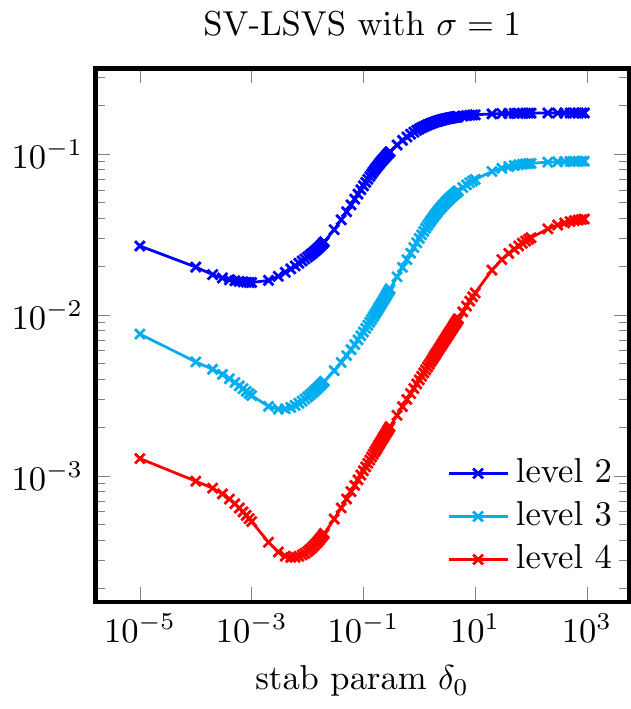}
 \caption{Example 2: $L^2$ velocity error for different stabilization parameters and different refinement levels for SV-LSVS (\(\sigma=0\)
    left and \(\sigma=1\) right) and fixed viscosity \(\mu=10^{-5}\).} \label{fig:param_extwo_lsvs}
\end{figure}

The SV-LSVS method on the other hand shows optimal convergence rates
for \(\sigma = 1\) (average EOC=$2.96)$ and delivers much smaller velocity errors on the finest mesh than the other two methods,
compare also the numbers in Tables~\ref{ex2_sigma0} and \ref{ex2_sigma1} for \(\sigma = 0\) and \(\sigma = 1\), respectively. 
Figure~\ref{fig:param_extwo_lsvs} shows a similar parameter study for SV-LSVS. 
One can see for both \(\sigma=0\) and \(\sigma=1\) that the optimal value lies between the interval \(10^{-2}\) to \(10^{-3}\).

\begin{table}[htb]\caption{Example 2: velocity and pressure errors for all methods and different refinement levels for \(\sigma =0\).}\label{ex2_sigma0}
	\begin{center}
		\begin{tabular}{l|lll|lll|lll|lll}
			ref & \multicolumn{3}{c|}{SV} & \multicolumn{3}{c|}{SV-SUPG} &
			\multicolumn{3}{c}{SV-LSVS}\\ 
			&$L^2(u)$& $H^1(u)$& $L^2(p)$&$L^2(u)$& $H^1(u)$ & $L^2(p)$ &$L^2(u)$ & $H^1(u)$ & $L^2(p)$ \\ \hline
			1&8.020e-1&19.86   &3.448e-1&1.179e-1&3.090   &7.888e-2& 1.681e-1&3.2900   & 2.213e-1\\
			2&1.420e-1&5.335   &6.186e-2&8.578e-2&1.903   &4.152e-2& 5.295e-2&1.0544   & 5.514e-2\\
			3&2.582e-2&2.682   &8.659e-3&1.911e-2&9.348e-1&8.968e-3& 1.058e-2&3.045e-1 & 1.180e-2\\
			4&2.668e-3&7.860e-1&1.291e-3&4.056e-3&3.888e-1&2.012e-3& 1.629e-3&8.472e-2 & 2.784e-3\\
			5&4.007e-4&1.832e-1&2.891e-4&5.303e-4&1.316e-1&3.333e-4& 1.858e-4&1.848e-2 & 6.697e-4\\
			\hline
			EOC & 2.74 & 1.69 & 2.55 & 1.95 & 1.14 & 1.97 & 2.46 & 1.87 & 2.09 \\
		\end{tabular}
	\end{center}
\end{table}

\begin{table}[htb]\caption{Example 2: velocity and pressure errors for all methods and different refinement levels for \(\sigma =1\).}\label{ex2_sigma1}
	\begin{center}
		\begin{tabular}{l|lll|lll|lll|lll}
			ref & \multicolumn{3}{c|}{SV} & \multicolumn{3}{c|}{SV-SUPG} &
			\multicolumn{3}{c}{SV-LSVS}\\ 
			&$L^2(u)$& $H^1(u)$& $L^2(p)$&$L^2(u)$& $H^1(u)$ & $L^2(p)$ &$L^2(u)$ & $H^1(u)$ & $L^2(p)$ \\ \hline
			1&1.790e-1&8.326   &9.088e-2&1.090e-1&2.923   &8.038e-2& 1.387e-1&3.1052  &2.222e-1\\
			2&3.367e-2&3.497   &2.152e-2&2.105e-2&1.277   &1.790e-2& 2.022e-2&8.847e-1&5.771e-2\\
			3&1.015e-2&1.900   &5.619e-3&5.501e-3&7.322e-1&4.364e-3& 2.751e-3&2.496e-1&1.264e-2\\
			4&1.679e-3&5.918e-1&1.142e-3&1.141e-3&3.306e-1&1.048e-3& 3.133e-4&6.505e-2&2.846e-3\\
			5&2.623e-4&1.638e-1&2.616e-4&2.194e-4&1.215e-1&2.550e-4& 3.741e-5&1.658e-2&6.775e-4\\
			\hline
			EOC & 2.35 & 1.42 & 2.11 & 2.24 & 1.15 & 2.08 & 2.96 & 1.89 & 2.09 \\
		\end{tabular}
	\end{center}
\end{table}

\subsection{Example 3: modified Planar lattice flow}
The third example takes the flow \(\bfu\) of Example 2 and modifies the right-hand side forcing such that \(\bfbeta = (0,1)^T\) and \(p = 0\).
 Note that this time \((\bfbeta \cdot \nabla) \bfu\) is a divergence-free field. Therefore, it is expected that
this example defines the best-case scenario for the SV-SUPG method due to $p=0$. In fact, this is the case, as SV-SUPG does improve the results
given by the plain Galerkin method, but still SV-LSVS provide a more accurate solution.

\begin{figure}[!htb]
	\centering
	\includegraphics[width=0.48\textwidth]{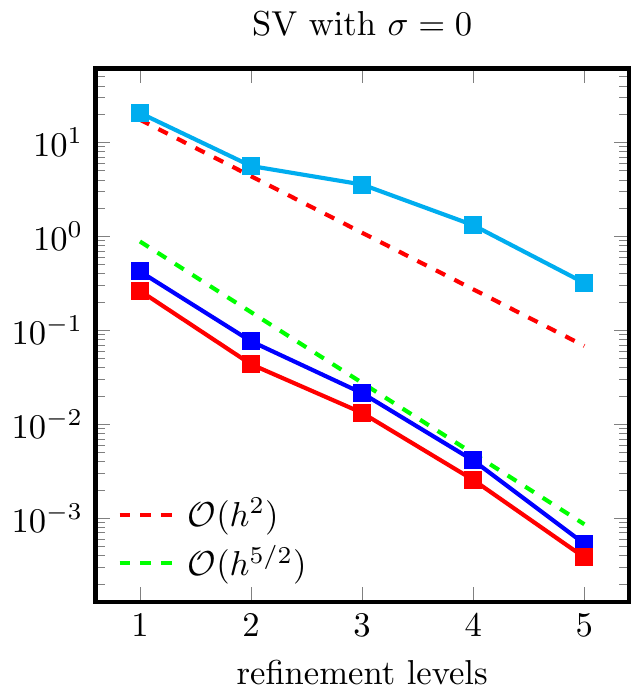}
	\includegraphics[width=0.48\textwidth]{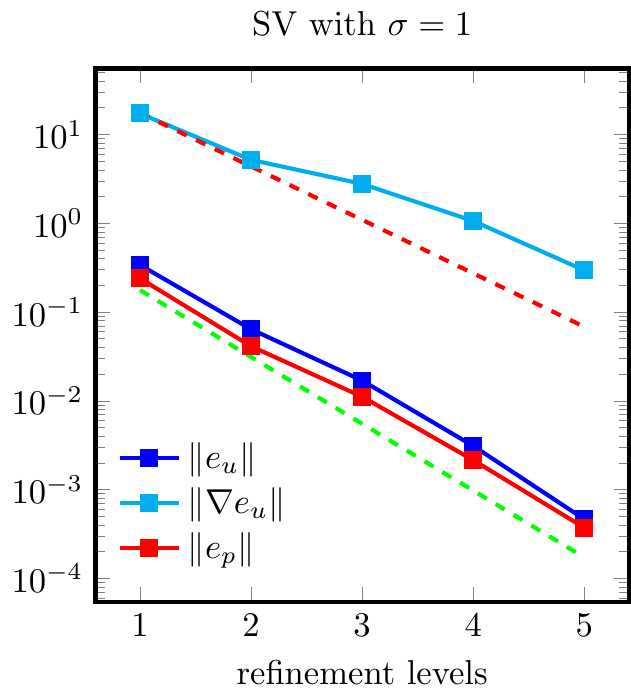}
	\caption{Example 3: error plots of different norms on different refinement levels 
	for Scott-Vogelius finite element methods (\(\sigma=0\)
	left and \(\sigma=1\) right) and fixed viscosity \(\mu=10^{-5}\).} \label{fig:nu_exthree_gal}
\end{figure}
\begin{figure}[!htb]
	\centering
	\includegraphics[width=0.48\textwidth]{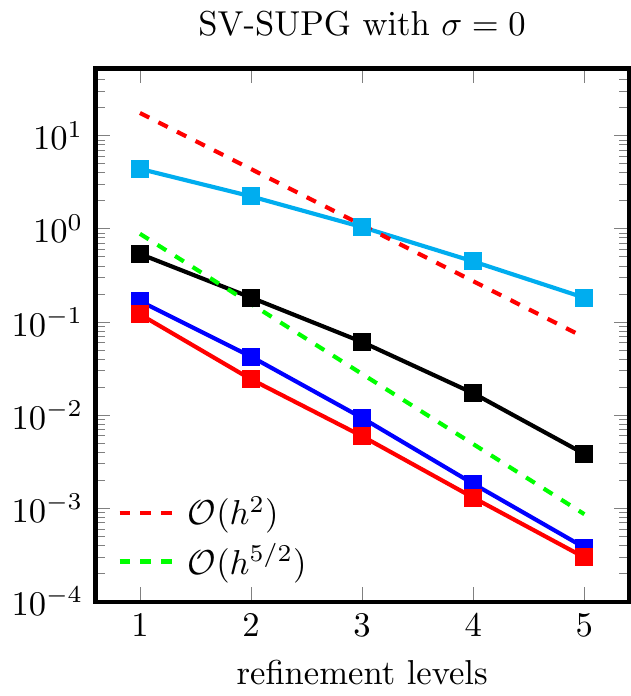}
	\includegraphics[width=0.48\textwidth]{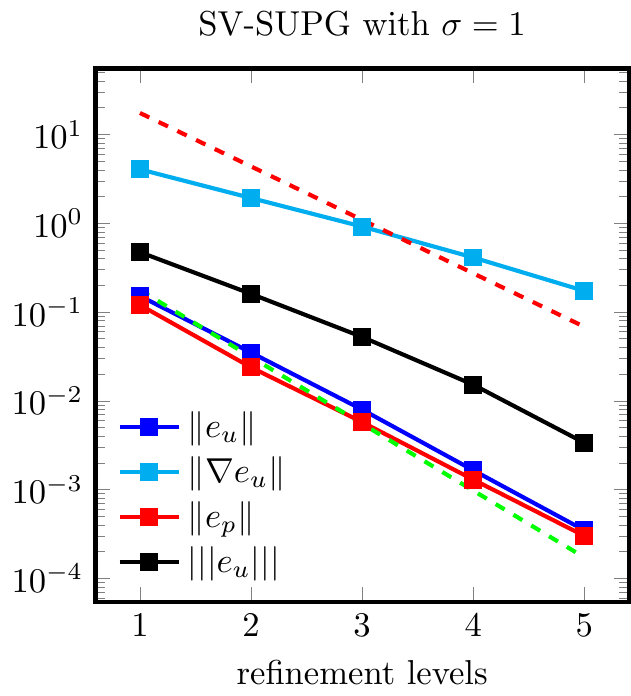}
	\caption{Example 3: error plots of different norms on different refinement levels
	for Scott-Vogelius with SUPG stabilization (\(\sigma=0\)
	left and \(\sigma=1\) right) and fixed viscosity \(\mu=10^{-5}\).} \label{fig:nu_exthree_supg}
\end{figure}
\begin{figure}[!htb]
\centering
\includegraphics[width=0.48\textwidth]{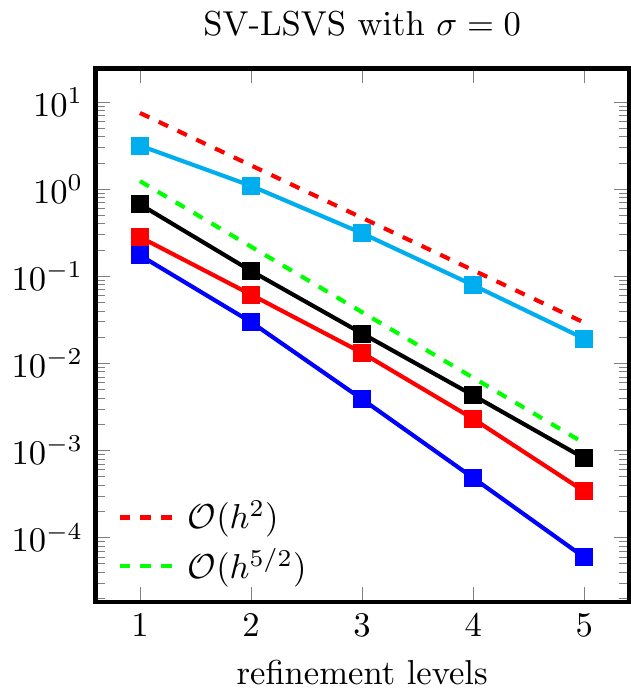}
\includegraphics[width=0.48\textwidth]{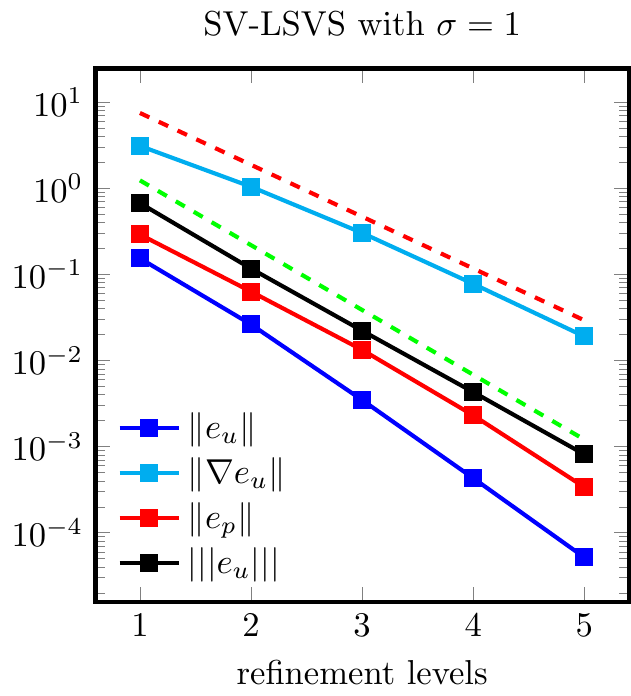}
 \caption{Example 3: error plots of different norms on different refinement levels
    for Scott-Vogelius with LSVS stabilization (\(\sigma=0\)
    left and \(\sigma=1\) right) and fixed viscosity \(\mu=10^{-5}\).} \label{fig:nu_exthree_conv}
\end{figure}

\begin{figure}[!htb]
\centering
\includegraphics[width=0.48\textwidth]{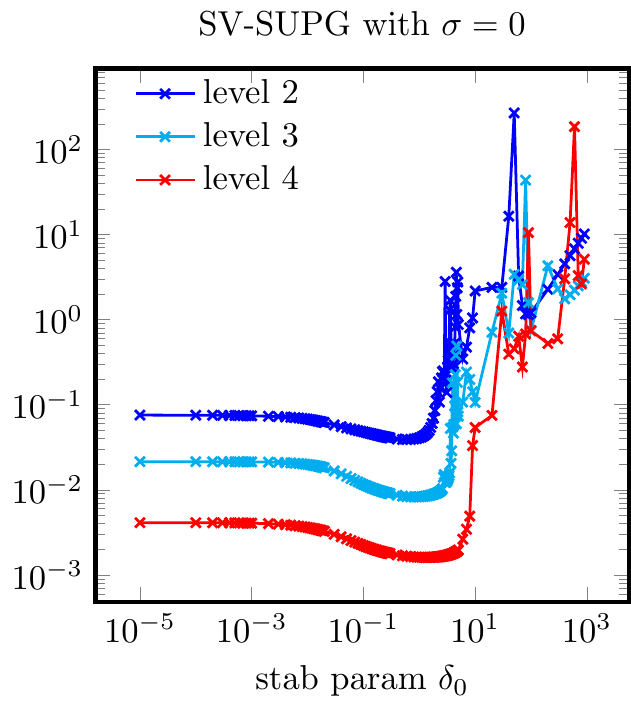}
\includegraphics[width=0.48\textwidth]{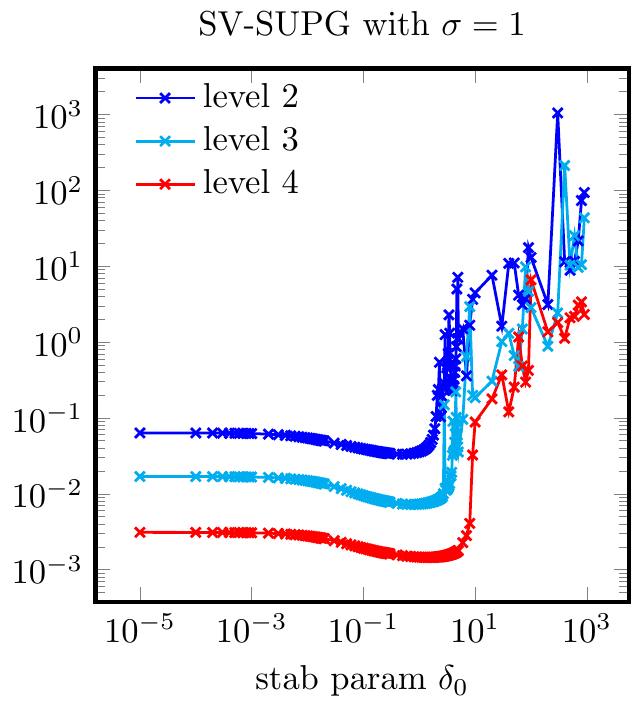}
 \caption{Example 3: $L^2$ velocity error for different stabilization parameters and different refinement levels for SV-SUPG (\(\sigma=0\)
    left and \(\sigma=1\) right) and fixed viscosity \(\mu=10^{-5}\).} \label{fig:param_ex3_supg}
\end{figure}

\begin{figure}[!htb]
\centering
\includegraphics[width=0.48\textwidth]{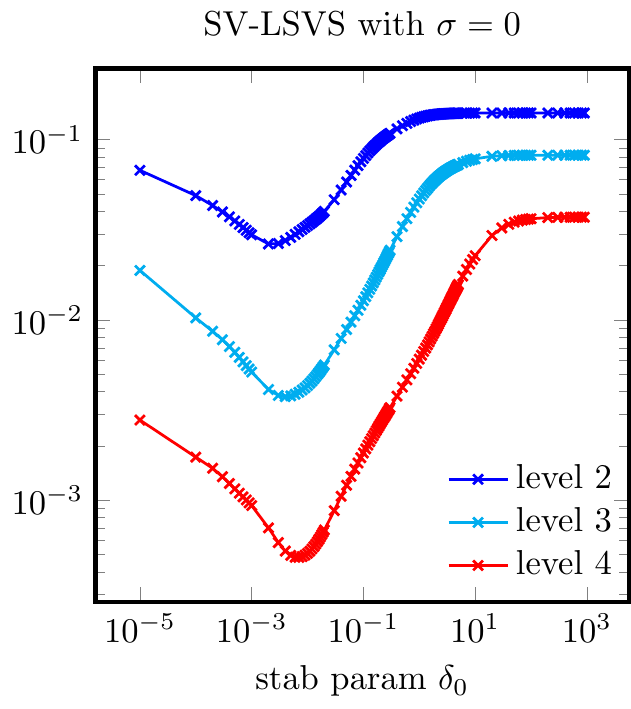}
\includegraphics[width=0.48\textwidth]{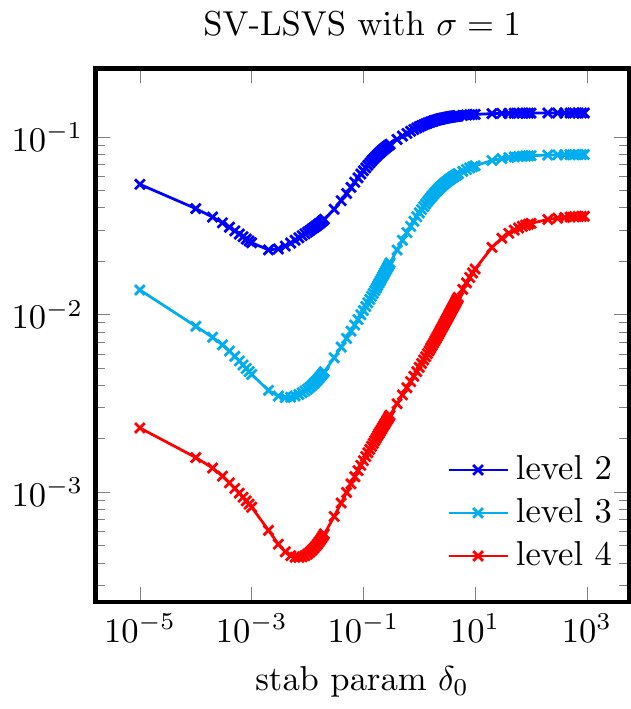}
 \caption{Example 3: $L^2$ velocity error for different stabilization parameters and different refinement levels for SV-LSVS (\(\sigma=0\)
    left and \(\sigma=1\) right) and fixed viscosity \(\mu=10^{-5}\).}  \label{fig:param_ex3_lsvs}
\end{figure}

\begin{table}[!h]\caption{Example 3: velocity and pressure errors for all methods and different refinement levels for \(\sigma = 0\).}\label{ex3_sigma0}
	\begin{center}
		\begin{tabular}{l|lll|lll|lll|lll}
			ref & \multicolumn{3}{c|}{SV} & \multicolumn{3}{c|}{SV-SUPG} &
			\multicolumn{3}{c}{SV-LSVS}\\ 
			&$L^2(u)$&$H^1(u)$& $L^2(p)$&$L^2(u)$& $H^1(u)$ & $L^2(p)$ &$L^2(u)$ & $H^1(u)$ & $L^2(p)$ \\ \hline
			1&4.237e-1&20.605   &2.640e-1&1.672e-1&4.398   &1.207e-1&1.742e-1&3.1664  &2.823e-1\\
			2&7.657e-2&5.6154   &4.357e-2&4.248e-2&2.228   &2.422e-2&2.982e-2&1.0913  &6.163e-2\\
			3&2.146e-2&3.5678   &1.323e-2&9.326e-3&1.041   &5.938e-3&3.875e-3&3.119e-1&1.320e-2\\
			4&4.124e-3&1.3164   &2.561e-3&1.832e-3&4.462e-1&1.300e-3&4.836e-4&7.899e-2&2.307e-3\\
			5&5.356e-4&3.1968e-1&3.835e-4&3.793e-4&1.818e-1&2.969e-4&5.916e-5&1.918e-2&3.389e-4\\
			\hline
			EOC & 2.41 & 1.50 & 2.36 & 2.20 & 1.15 & 2.17 & 2.88 & 1.84 & 2.43 \\
			\end{tabular}
	\end{center}
\end{table}
\begin{table}[!h]\caption{Example 3: velocity and pressure errors for all methods and different refinement levels for \(\sigma = 1\).}\label{ex3_sigma1}
	\begin{center}
		\begin{tabular}{l|lll|lll|lll|lll}
			ref & \multicolumn{3}{c|}{SV} & \multicolumn{3}{c|}{SV-SUPG} &
			\multicolumn{3}{c}{SV-LSVS}\\ 
			&$L^2(u)$&$H^1(u)$& $L^2(p)$&$L^2(u)$& $H^1(u)$ & $L^2(p)$ &$L^2(u)$ & $H^1(u)$ & $L^2(p)$ \\ \hline
			1&3.397e-1&17.27   &2.402e-1&1.518e-1&4.056   &1.203e-1&1.536e-1&3.1088  &2.911e-1\\
			2&6.418e-2&5.188   &4.125e-2&3.504e-2&1.927   &2.386e-2&2.626e-2&1.0425  &6.320e-2\\
			3&1.694e-2&2.781   &1.115e-2&7.981e-3&9.191e-1&5.768e-3&3.483-e3&3.033e-1&1.331e-2\\
			4&3.107e-3&1.062   &2.152e-3&1.654e-3&4.122e-1&1.298e-3&4.308e-4&7.772e-2&2.310e-3\\
			5&4.646e-4&2.952e-1&3.725e-4&3.490e-4&1.732e-1&3.014e-4&5.178e-5&1.905e-2&3.390e-4\\
		    \hline
			EOC & 2.38 & 1.47 & 2.33 & 2.19 & 1.14 & 2.16 & 2.88 & 1.84 & 2.44 \\
		\end{tabular}
	\end{center}
\end{table}
Tables~\ref{ex3_sigma0} and \ref{ex3_sigma1} confirm this expectation that the SV-SUPG method works as well as the SV-LSVS method. One can see that the SV-SUPG method converges optimally. However, the SV-LSVS method
delivers a
slightly better velocity than the SV-SUPG method (a factor 6 smaller on the finest mesh). Figure~\ref{fig:param_ex3_supg} confirms that SV-SUPG method works close to its optimum with the default parameter \(\delta_0 = 0.25\). 
 Figure~\ref{fig:param_ex3_lsvs} for the SV-LSVS method on the other hand shows that  \(\delta_0 = 0.006\) is a good estimate for  optimal parameter value.
 
\subsection{Example 4: 'superposition' of Example 2 and 3}

The last example combines the flows of Examples 2 and 3 and employs a superposition of their convective forces. This is, the convective term is given by \(\bfbeta := \bfu + (0,1)^T\), while \(\bfu\) and \(p\) are the same as in Example 2. This is somehow considered to be a 'realistic' situation where the (discrete and asymptotic) convective forcing has an irrotational part (as in Examples 1 and 2) and a divergence-free part (as in Example 3).

\begin{figure}[!htb]
	\centering
	\includegraphics[width=0.48\textwidth]{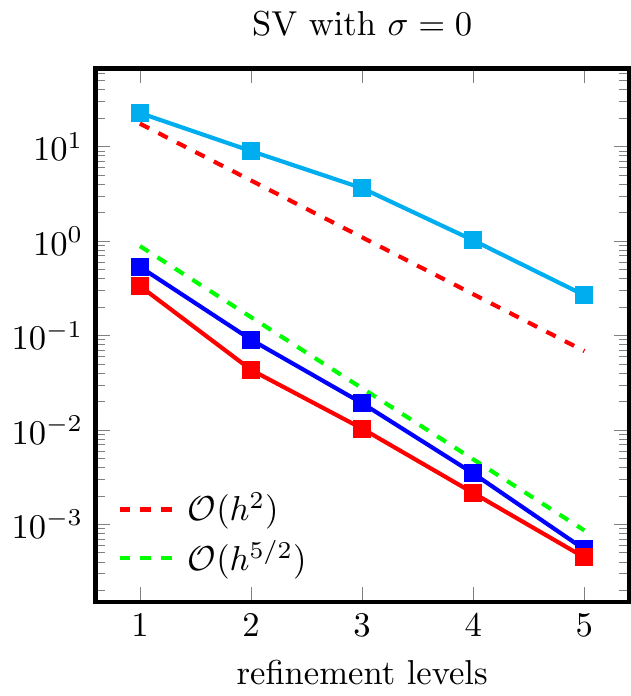}
	\includegraphics[width=0.48\textwidth]{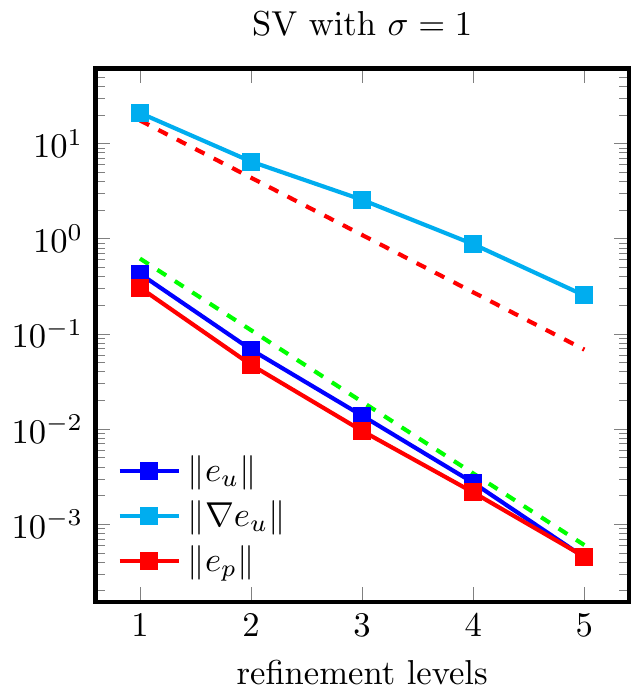}
	\caption{Example 4: error plots of different norms on different refinement levels 
	for Scott-Vogelius finite element methods (\(\sigma=0\)
	left and \(\sigma=1\) right) and fixed viscosity \(\mu=10^{-5}\).} \label{fig:nu_exfour_gal}
\end{figure}

\begin{figure}[!htb]
	\centering
	\includegraphics[width=0.48\textwidth]{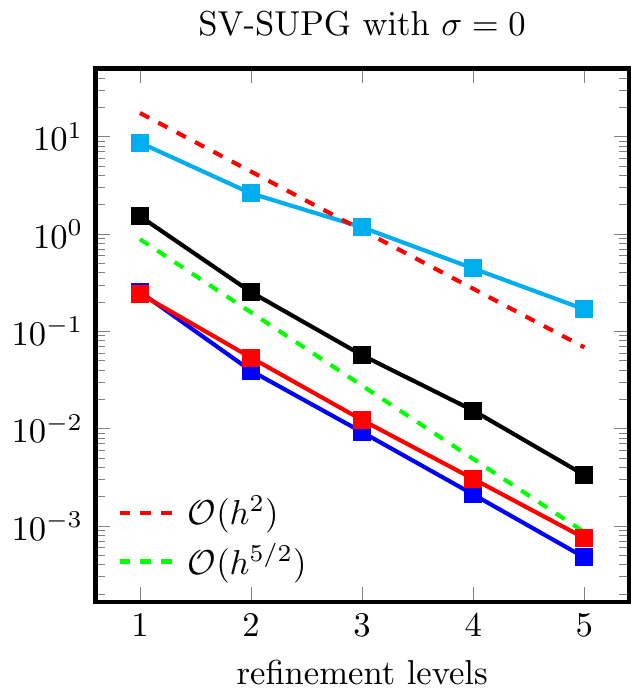}
	\includegraphics[width=0.48\textwidth]{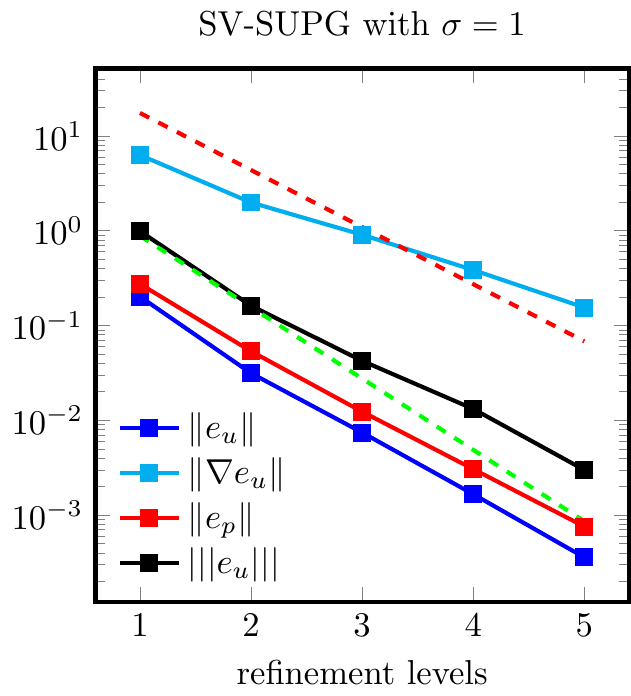}
	\caption{Example 4: error plots of different norms on different refinement levels
	for Scott-Vogelius with SUPG stabilization (\(\sigma=0\)
	left and \(\sigma=1\) right) and fixed viscosity \(\mu=10^{-5}\).} \label{fig:nu_exfour_supg}
\end{figure}
\begin{figure}[!htb]
\centering
\includegraphics[width=0.48\textwidth]{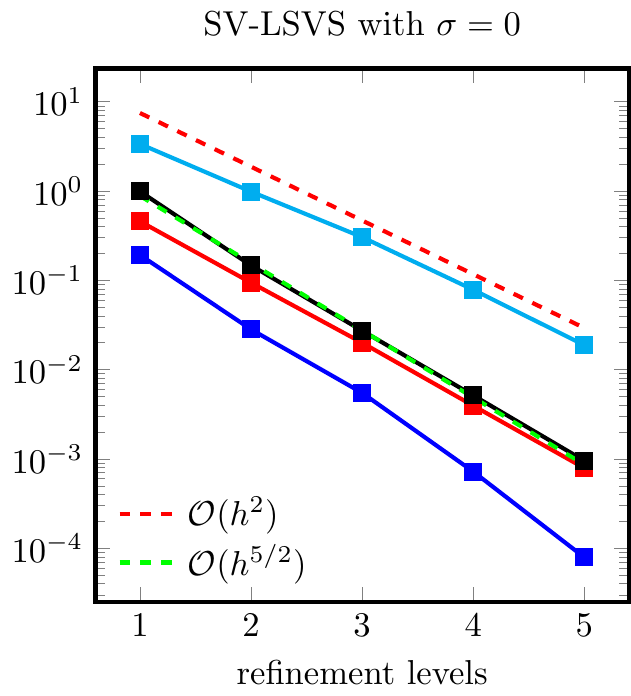}
\includegraphics[width=0.48\textwidth]{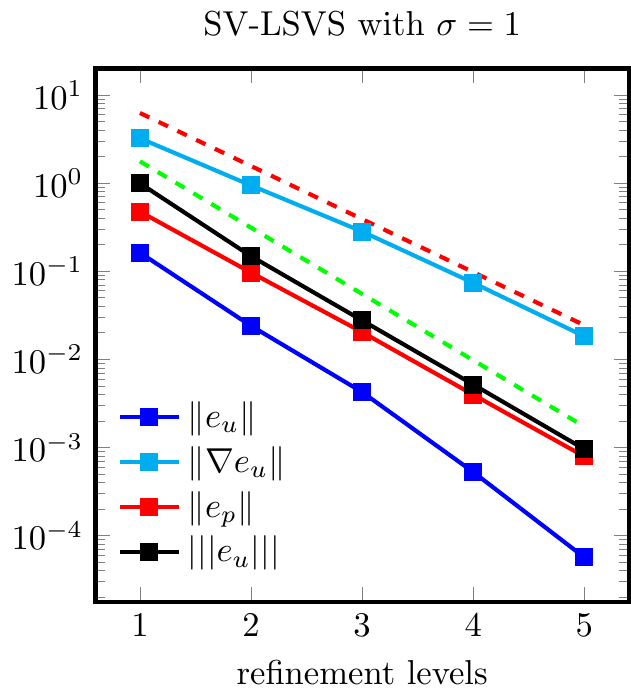}
 \caption{Example 4: error plots of different norms on different refinement levels
    for Scott-Vogelius with convection stabilization (\(\sigma=0\)
    left and \(\sigma=1\) right) and fixed viscosity \(\mu=10^{-5}\).} \label{fig:nu_exfour_lsvs}
\end{figure}

\begin{figure}[!htb]
\centering
\includegraphics[width=0.48\textwidth]{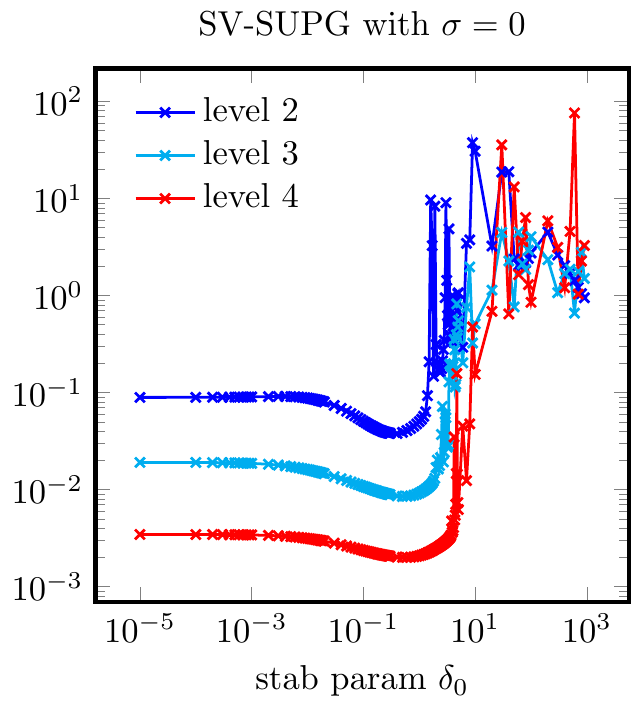}
\includegraphics[width=0.48\textwidth]{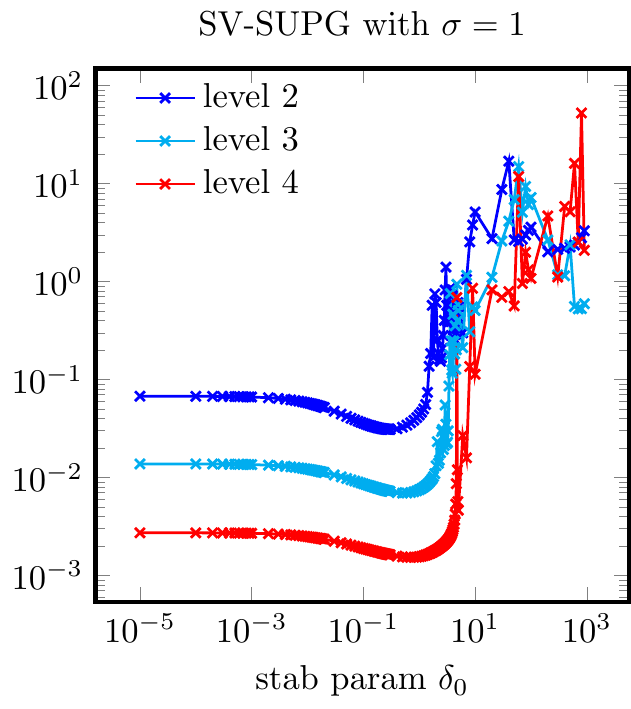}
 \caption{Example 4: $L^2$ velocity error for different stabilization parameters and different refinement levels for SV-SUPG (\(\sigma=0\)
    left and \(\sigma=1\) right) and fixed viscosity \(\mu=10^{-5}\).} \label{fig:param_ex4_supg}
\end{figure}

\begin{figure}[!htb]
\centering
\includegraphics[width=0.48\textwidth]{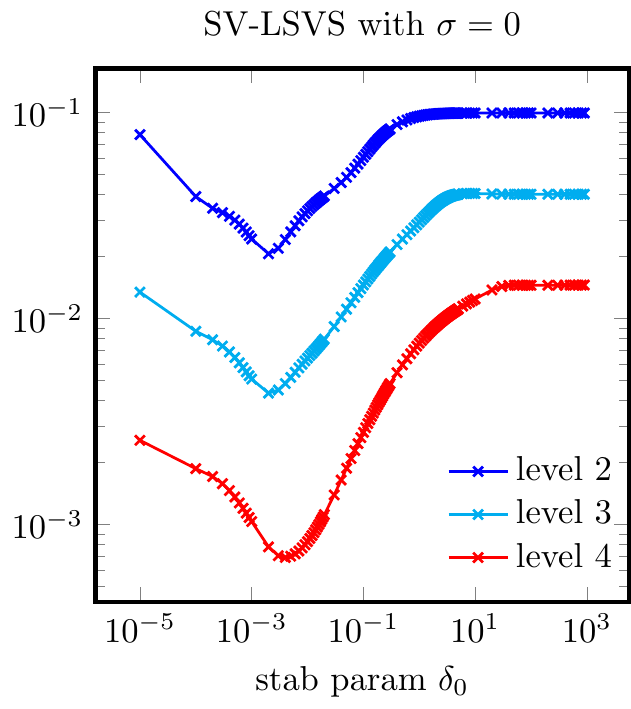}
\includegraphics[width=0.48\textwidth]{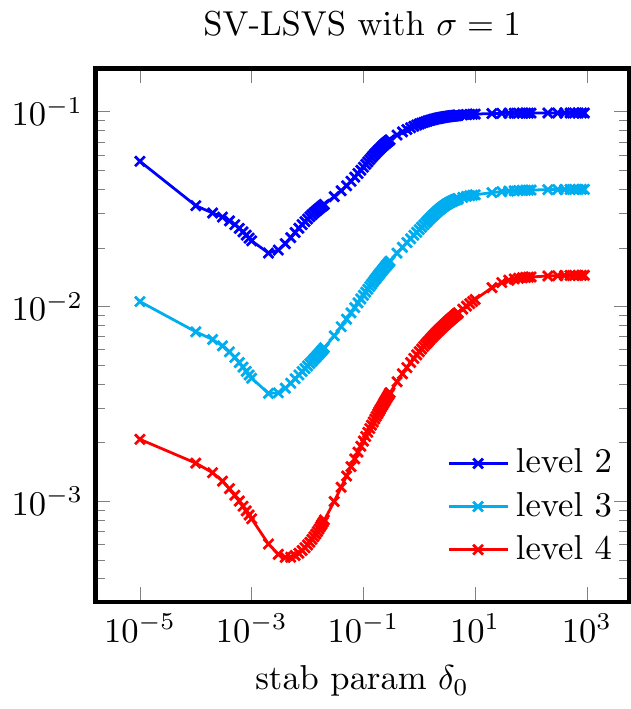}
 \caption{Example 4: $L^2$ velocity error for different stabilization parameters and different refinement levels for SV-LSVS (\(\sigma=0\)
    left and \(\sigma=1\) right) and fixed viscosity \(\mu=10^{-5}\).
    }  \label{fig:param_ex4_lsvs}
\end{figure}
\begin{table}[!h]\caption{Example 4: velocity and pressure errors for all methods and different refinement levels for \(\sigma = 0\).}\label{ex4_sigma0}
	\begin{center}
		\begin{tabular}{l|lll|lll|lll|lll}
			ref & \multicolumn{3}{c|}{SV} & \multicolumn{3}{c|}{SV-SUPG} &
			\multicolumn{3}{c}{SV-LSVS}\\ 
			&$L^2(u)$&$H^1(u)$& $L^2(p)$&$L^2(u)$& $H^1(u)$ & $L^2(p)$ &$L^2(u)$ & $H^1(u)$ & $L^2(p)$ \\ \hline
			1&5.328e-1&2.269e+1&3.339e-1&2.540e-1&8.6706  &2.434e-1& 1.929e-1&3.3734  &4.587e-1\\
			2&9.032e-2&8.969e+0&4.330e-2&3.928e-2&2.6116  &5.363e-2& 2.826e-2&9.865e-1&9.425e-2\\
			3&1.919e-2&3.627e+0&1.033e-2&9.198e-3&1.1702  &1.224e-2& 5.499e-3&3.050e-1&2.000e-2\\
			4&3.467e-3&1.016e+0&2.150e-3&2.107e-3&4.421e-1&3.043e-3& 7.221e-4&7.851e-2&3.949e-3\\
			5&5.443e-4&2.668e-1&4.473e-4&4.752e-4&1.680e-1&7.519e-4& 7.904e-5&1.882e-2&7.901e-4\\
			\hline
			EOC & 2.48 & 1.60 & 2.39 & 2.27 & 1.42 & 2.08 & 2.81 & 1.87 & 2.30 \\
		\end{tabular}
	\end{center}
\end{table}
\begin{table}[!h]\caption{Example 4: velocity and pressure errors for all methods and different refinement levels for \(\sigma = 1\).}\label{ex4_sigma1}
	\begin{center}
		\begin{tabular}{l|lll|lll|lll|lll}
			ref & \multicolumn{3}{c|}{SV} & \multicolumn{3}{c|}{SV-SUPG} &
			\multicolumn{3}{c}{SV-LSVS}\\ 
			&$L^2(u)$&$H^1(u)$& $L^2(p)$&$L^2(u)$& $H^1(u)$ & $L^2(p)$ &$L^2(u)$ & $H^1(u)$ & $L^2(p)$ \\ \hline
			1&4.284e-1&2.093e+1&3.066e-1&2.022e-1&6.3029   &2.751e-1&1.624e-1&3.2283  &4.747-1\\
			2&6.847e-2&6.468e+0&4.734e-2&3.146e-2&1.9888   &5.358e-2&2.399e-2&9.408e-1&9.650-2\\
			3&1.382e-2&2.562e+0&9.569e-3&7.439e-3&9.080e-1 &1.234e-2&4.255e-3&2.822e-1&2.036-2\\
			4&2.729e-3&8.759e-1&2.161e-3&1.663e-3&3.854e-1 &3.062e-3&5.264e-4&7.382e-2&3.970-3\\
			5&4.487e-4&2.533e-1&4.561e-4&3.585e-4&1.550e-1 &7.579e-4&5.662e-5&1.826e-2&7.907-4\\
						\hline
			EOC & 2.47 & 1.59 & 2.35 & 2.28 & 1.34 & 2.13 & 2.87 & 1.87 & 2.31 \\
		\end{tabular}
	\end{center}
\end{table}

As expected from the experience with the other examples, both stabilization methods significantly improve the errors compared to the plain SV method. There is also a clear improvement of SV-LSVS compared to SV-SUPG.
Only SV-LSVS has an optimal convergence behavior,
compare Figures~\ref{fig:nu_exfour_gal}-\ref{fig:nu_exfour_lsvs}
and Tables~\ref{ex4_sigma0} and \ref{ex4_sigma1}.

\section{Concluding remarks} \label{concl}

In this work a new stabilized finite element method for the Oseen has been proposed and analyzed. The method
is based on the observation that, in order to obtain pressure-robust error estimates, the stabilization term needs
to be independent of the pressure. That is why the stabilizing term is built as a penalization of the vorticity equation,
where the pressure gradient is not present. This design has allowed us to prove optimal, pressure-independent error
estimates for the velocity. In particular, the $O(h^{k+\frac{1}{2}})$ error bound for $\|\bfu-\bfu_h^{}\|_{0,\Omega}^{}$, not
available for the Galerkin method or the SUPG method when applied to inf-sup stable discretizations, and also only available 
so far for $H^1$-conforming equal order stabilized methods (at the price of a constant that depends on the regularity of the pressure).
From the numerical results we can extract the following conclusions:
\begin{itemize}
    \item SV-LSVS works well and converges with an optimal order in any situation (Example 1-4); in the extreme
    Example 1 it delivers the
    exact solution for every stabilization parameter;
    \item SV-SUPG converges always sub-optimally. In situations,
    where the convective force is close to a gradient it can be less
    accurate than the plain SV method. However, for situations,
    where the convective term is divergence-free, SV-SUPG delivers
    more accurate results on coarse meshes than plain SV;
    \item SV-LSVS outperforms plain SV and SV-SUPG,
    in the most general Example 4, where the convective term
    has a divergence-free and an irrotational part;
    \item the SV-LSVS has a robust behavior with respect
    to the stabilization parameter. For all the Examples 1--4, the
    same parameter $\delta_0^{}=0.006$  was used. Instead, for SV-SUPG in Example 1
    it could be shown that the optimal parameter is $\delta_0^{}=0$,
    while it is about $\delta_0^{}\approx0.25$ for Examples 2-4.
\end{itemize}

\section*{Acknowledgements} The work of GRB has been funded by the Leverhulme Trust through the Research Fellowship No. RF-2019-510.

\bibliographystyle{abbrv}
\bibliography{divfree}

\end{document}